\documentclass[openany]{amsart}
\usepackage{amssymb, amsfonts, lacromay}
\usepackage{mathrsfs,comment}
\usepackage[usenames,dvipsnames]{color}
\usepackage[normalem]{ulem}
\usepackage{url}

\usepackage[all,arc,2cell]{xy}
\UseAllTwocells
\usepackage{enumerate}

\usepackage{hyperref}  
\hypersetup{%
  bookmarksnumbered=true,%
  bookmarks=true,%
  colorlinks=true,%
  linkcolor=blue,%
  citecolor=blue,%
  filecolor=blue,%
  menucolor=blue,%
  pagecolor=blue,%
  urlcolor=blue,%
  pdfnewwindow=true,%
  pdfstartview=FitBH}

\def\makeautorefname#1#2{\expandafter\def\csname#1autorefname\endcsname{#2}}
\def\equationautorefname~#1\null{(#1)\null}
\makeautorefname{footnote}{footnote}%
\makeautorefname{item}{item}%
\makeautorefname{figure}{Figure}%
\makeautorefname{table}{Table}%
\makeautorefname{part}{Part}%
\makeautorefname{appendix}{Appendix}%
\makeautorefname{chapter}{Chapter}%
\makeautorefname{section}{Section}%
\makeautorefname{subsection}{Section}%
\makeautorefname{subsubsection}{Section}%
\makeautorefname{theorem}{Theorem}%
\makeautorefname{thm}{Theorem}%
\makeautorefname{cor}{Corollary}%
\makeautorefname{lem}{Lemma}%
\makeautorefname{prop}{Proposition}%
\makeautorefname{pro}{Property}
\makeautorefname{conj}{Conjecture}%
\makeautorefname{defn}{Definition}%
\makeautorefname{notn}{Notation}
\makeautorefname{notns}{Notations}
\makeautorefname{rem}{Remark}%
\makeautorefname{exmp}{Example}%
\makeautorefname{ax}{Axiom}%
\makeautorefname{claim}{Claim}%
\makeautorefname{ass}{Assumption}%
\makeautorefname{asses}{Assumptions}%
\makeautorefname{con}{Construction}%
\makeautorefname{prob}{Problem}%
\makeautorefname{warn}{Warning}%

\newtheorem{thm}{Theorem}[section]
\newtheorem{cor}{Corollary}[section]
\newtheorem{prop}{Proposition}[section]
\newtheorem{lem}{Lemma}[section]
\newtheorem{prob}{Problem}[section]

\theoremstyle{definition}
\newtheorem{defn}{Definition}[section]
\newtheorem{ass}{Assumption}[section]
\newtheorem{ax}{Axiom}[section]
\newtheorem{con}{Construction}[section]
\newtheorem{exmp}{Example}[section]
\newtheorem{notn}{Notation}[section]

\newtheorem{rem}{Remark}[section]
\newtheorem{warn}{Warning}[section]

\makeatletter
\let\c@cor=\c@thm
\let\c@prop=\c@thm
\let\c@lem=\c@thm
\let\c@prob=\c@thm
\let\c@con=\c@thm
\let\c@conj=\c@thm
\let\c@defn=\c@thm
\let\c@notn=\c@thm
\let\c@notns=\c@thm
\let\c@exmp=\c@thm
\let\c@ax=\c@thm
\let\c@pro=\c@thm
\let\c@ass=\c@thm
\let\c@warn=\c@thm
\let\c@rem=\c@thm
\let\c@sch=\c@thm
\let\c@equation\c@thm
\numberwithin{equation}{section}
\makeatother

\definecolor{orange}{rgb}{1,0.5,0}

\renewcommand{\bf}[1]{\mathbf{#1}}

\newcommand{\cE}{\mathcal{E}}

\newcommand{\opa}{\sO\text{-}\bf{PsAlg}}
\newcommand{\oap}{\sO\text{-}\bf{AlgPs}}
\newcommand{\oas}{\sO\text{-}\bf{AlgSt}}

\newcommand{\tpa}{\bT\text{-}\bf{PsAlg}}
\newcommand{\etp}{\bT\text{-}\bf{AlgPs}}
\newcommand{\ets}{\bT\text{-}\bf{AlgSt}}

\title{Symmetric monoidal $G$-categories and their strictification}

\author{B. Guillou}
\address{Department of Mathematics, The University of Kentucky, Lexington, KY 40506--0027}
\email{bertguillou@uky.edu}
\author{J.P. May}
\address{Department of Mathematics, The University of Chicago, Chicago, IL 60637}
\email{may@math.uchicago.edu}
\author{M. Merling}
\address{department of Mathematics, The University of Pennsylvania, Philadelphia, PA 19104}
\email{mmerling@math.upenn.edu}
\author{A.M. Osorno}
\address{Department of Mathematics, Reed College, Portland, OR 97202}
\email{aosorno@reed.edu}

\thanks{B. Guillou was  supported  by Simons Collaboration Grant  282316 and NSF grant DMS-1710379.}\thanks{M. Merling was supported by NSF grant DMS-1709461}
\thanks{A.M. Osorno was  supported by Simons Foundation Grant  359449, the Woodrow Wilson Career Enhancement Fellowship and NSF grant DMS-1709302}
\thanks{NSF RTG grant DMS-1344997 supported several collaborator visits to Chicago.}


\begin{document}

\begin{abstract} 
We give an operadic definition of a \emph{genuine} symmetric monoidal $G$-category, and we prove 
that its classifying space is a genuine $E_\infty$ $G$-space. 
We do this by developing some very general categorical coherence theory.  We combine results of Corner and Gurski, Power, and Lack, to develop a strictification theory for pseudoalgebras over operads and monads.  It specializes to strictify genuine symmetric monoidal $G$-categories to genuine permutative $G$-categories. All of our work takes place in a general internal categorical framework that has many quite different specializations. 

When $G$ is a finite group, the theory here
combines with previous work to generalize equivariant infinite loop space theory from strict 
space level input to considerably more general category level input.   It takes genuine 
symmetric monoidal $G$-categories as input to an equivariant infinite loop space machine that 
gives genuine $\OM$-$G$-spectra as output.
\end{abstract}

\maketitle

\tableofcontents

\setcounter{page}{2}

\section*{Introduction and statements of results}
Symmetric monoidal categories are fundamental to much of mathematics, and they provide crucial input to the infinite loop space theory developed in the early 1970's.   There it was very convenient to use the still earlier categorical strictification theory showing that symmetric monoidal categories are monoidally equivalent to symmetric strict monoidal categories, whose products are strictly associative and unital. Following Anderson \cite{And}, topologists call symmetric strict monoidal categories ``permutative categories''.  

Equivariantly, we take this as inspiration, and in this paper we give a definition of \emph{genuine} symmetric monoidal $G$-categories and prove that they can be strictified to genuine permutative $G$-categories, as defined in \cite{GM3}.    These are $G$-categories with extra structure that ensures that their classifying spaces are genuine $E_\infty$ $G$-spaces, so that after equivariant group completion they can be delooped by any finite dimensional representation $V$ of $G$.  This theory shows that we can construct genuine $G$-spectra and maps between them from genuine symmetric monoidal 
$G$-categories and functors that respect the monoidal structure only up to isomorphism.    

While this paper is a spin-off from a large scale ongoing project on equivariant infinite loop space theory, it gives a reasonably self-contained exposition of the relevant categorical coherence theory.  In contrast to its equivariant setting in our  larger project, this work is designed to be more widely applicable, and in fact the equivariant setting plays no particular role other than providing motivation.  We say more about that motivation shortly, but we first discuss the categorical context in which most of our work takes place. 

Category theorists have developed a powerful and subtle theory of $2$-monads and their pseudoalgebras  \cite{BKP,  Lack, Power, StMon}. It gives just the right framework and results for our strictification theorem.  Working in an arbitrary ground $2$-category $\sK$, 
we briefly recall the definitions of $2$-monads  $\bT$, (strict) $\bT$-algebras and $\bT$-pseudoalgebras, (strict) $\bT$-maps and $\bT$-pseudomorphisms, and  algebra $2$-cells in \autoref{Monpseudo}. With these definitions, we have the following three $2$-categories.\footnote{We shall make no use of the second choice.  We include it because it is often convenient and much of the relevant categorical literature focuses on it.}

\vspace{1.5mm}

$\bullet$ $\tpa$: $\bT$-pseudoalgebras and $\bT$-pseudomorphisms.

\vspace{2mm}

$\bullet$ $\etp$: $\bT$-algebras and $\bT$-pseudomorphisms.

\vspace{2mm}

$\bullet$ $\ets$: $\bT$-algebras and (strict) $\bT$-maps.

\vspace{1.5mm}

\noindent
In all of them, the $2$-cells are the algebra 2-cells.

Power discovered \cite{Power} and Lack elaborated \cite{Lack} a remarkably simple way to strictify structures over a 
$2$-monad.\footnote{We are greatly indepted to Power and Lack for correspondence about this result.} 
Power's short paper defined the strictification $\bf{St}$ on pseudo\-algebras, and Lack's  short paper (on codescent objects) defined $\bf{St}$ on $1$-cells and $2$-cells.  The result and its proof are truly beautiful category theory.  Generalizing to our internal categorical context, we obtain the following strictification theorem in \autoref{PL1}.

\begin{thm}\label{Lack0} Let $\sK$ have a rigid enhanced factorization system $(\mathcal{E},\sM)$ and let $\bT$ be a monad in $\sK$ which preserves $\cE$.
Then the inclusion of $2$-categories
\[ \bJ\colon \ets \rtarr \tpa  \]
has a left 2-adjoint strictification $2$-functor $\mathbf{St}$, and the component of the unit 
 of the adjunction is an internal equivalence in $\tpa$.
\end{thm}

As we explain in \autoref{counitissue}, the counit also becomes an internal equivalence once we use $\bJ$ to consider it as a map of pseudo-algebras.

We shall take the opportunity to expand on the papers of Power and Lack with a number of new details, and we give a reasonably complete and self-contained exposition.  The hypothesis about rigid enhanced factorization systems (EFS) is developed and specialized to the examples of interest to us in \autoref{EFS0}, and the construction of $\bf{St}$ and proof of the theorem are given in \autoref{PL}.

The reader is forgiven if she does not immediately see a connection between this theorem and our motivation in terms of symmetric monoidal $G$-categories.  That is what the rest of the paper provides.
 Our focus is on the $2$-category $\sK = \bf{Cat}(\sV)$ of categories internal to a suitable category 
 $\sV$.  We describe this context in \autoref{Pre1}.  We specify a rigid EFS on $\bf{Cat}(\sV)$ in \autoref{EFS2}, deferring proofs to \autoref{EFS3}.\footnote{We are greatly indepted to Gurski for correspondence about this generalization of the EFT on $\bf{Cat}$ defined by Power \cite{Power}.}  This has nothing to do with operads or monads.
 
As we show in \autoref{OpMon}, an operad $\sO$ in $\bf{Cat}(\sV)$ has an associated $2$-monad 
$\bO$ defined on $\bf{Cat}(\sV)$.  Guided by the monadic theory and largely following Corner and Gurski \cite{CG}, we define $\sO$-pseudoalgebras, $\sO$-pseudomorphisms, and algebra $2$-cells 
(alias $\sO$-transformations) in \autoref{Oppseudo}.  With these definitions, we have the three $2$-categories

\vspace{1.5mm}

$\bullet$ $\opa$: $\sO$-pseudoalgebras and $\sO$-pseudomorphisms.

\vspace{2mm}

$\bullet$ $\oap$: $\sO$-algebras and $\sO$-pseudomorphisms.

\vspace{2mm}

$\bullet$ $\oas$: $\sO$-algebras and (strict) $\sO$-maps.

\vspace{1.5mm}

\noindent
In all of them, the $2$-cells are the algebra 2-cells.

With motivation from symmetric monoidal categories, our definitions in \autoref{pseudo} differ a bit from those in the literature, in particular adding normality conditions.  We have tailored our definitions so that an immediate comparison gives the following monadic identifications of our $2$-categories of operadic algebras in $\bf{Cat}(\sV)$.

\vspace{1.5mm}

\begin{center}

 $\sO\text{-}\bf{PsAlg}= \bO\text{-}\bf{PsAlg} $

\vspace{2mm}

 $\sO\text{-}\bf{AlgPs} =  \bO\text{-}\bf{AlgPs}$
\vspace{2mm}

 $\sO\text{-}\bf{AlgSt} = \bO\text{-}\bf{AlgSt} $ 
 
 \end{center}
 
 \vspace{1.5mm}
 
It requires some work to define $\opa$ since $\bf{Cat}(\sV)$ is a 2-category, so that instead of requiring the usual diagrams in the strict context to commute, we must fill them with 2-cells that are required to be coherent and we must make the coherence precise. The monadic forerunner charts the path.

Of course, this is analogous to the identification of  $\sO$-algebras and $\bO$-algebras for operads in spaces that  motivated the coinage of the word ``operad'' in the first place \cite{MayGeo}.  The theory of $2$-monads gives a formalism that allows us to treat operad algebras in a context with many other examples. It will be applied to algebras over categories of operators in a sequel.

With these identifications, \autoref{Lack} has the following specialization.
 
 \begin{thm}\label{ConjOut1} Let $\sO$ be an operad in $\bf{Cat}(\sV)$.  Then the inclusion of $2$-categories 
$$ \bJ\colon \oas \rtarr \opa $$
has a left 2-adjoint 
$$\mathbf{St} \colon \opa \rtarr \oas,$$
and the components of the unit 
 of the adjunction are internal equivalences in $\opa$.
\end{thm}

Returning to our motivation, we discuss the specialization to symmetric monoidal categories in \autoref{genuine}. Nonequivariantly, permutative categories are the same thing as $\sP$-algebras in $\mathbf{Cat}$, where $\sP$ is the categorical version of the 
Barratt-Eccles operad. Formally, the category of permutative  categories is isomorphic to the 
category of $\sP$-algebras in $\mathbf{Cat}$ \cite{MayPerm}.  This suggests a generalization in which we replace $\sP$ by a more general operad and replace $\mathbf{Cat}$ by a more general category of (small) categories.  The generalization is illuminating nonequivariantly and should have other applications, but it is essential equivariantly, as we now explain. 

A naive permutative $G$-category is a permutative category with $G$-action, that is, a $G$-category with an action of the operad  $\sP$, where we think of the categories $\sP(j)$ as $G$-categories with trivial $G$-action.  Nonequivariantly, permutative categories are the input of an operadic infinite loop space machine defined in \cite{MayPerm, Seg} and axiomatized in \cite{MayPerm2}.  Its output is connective $\OM$-spectra with zeroth space given by the group completion of the classifying space of the input permutative category. Naive permutative $G$-categories work the same way.  They naturally give rise to naive $\OM$-$G$-spectra, which are $\OM$-spectra with $G$-action.  They represent $\mathbb{Z}$-graded equivariant cohomology theories, such as classical Borel and Bredon cohomology.  As is well-known, these are inadequate for applications.  In particular, no version of Poincar\'e duality  for $G$-manifolds holds in them.   For that, one must work in cohomology theories graded on representations (e.g. \cite[Chapter~III]{LMS}), and those are represented by \emph{genuine} $\OM$-$G$-spectra. The zeroth space of such a spectrum has deloopings not only for all spheres $S^n$, but also for all representation spheres $S^V$.

Genuine permutative $G$-categories are defined in \cite{ GM3} as algebras over an equivariant generalization $\sP_G$ of $\sP$ , and these give the input for an operadic equivariant infinite loop space machine that produces genuine $\OM$-$G$-spectra. We do not know any interpretation of genuine permutative $G$-categories other than the operadic one.  Since the operads $\sP$ and  $\sP_G$ are the ones whose 
algebras are permutative categories, we call them the {\em permutativity operads} henceforward, and we recall their definitions in \S\ref{genuine}.

Morphisms between symmetric monoidal categories, or even between permutative categories, are rarely strict; they are given by strong and sometimes even lax symmetric monoidal functors.   Classical coherence theory shows how to convert such morphisms of symmetric monoidal categories to symmetric strict monoidal functors between permutative categories.  By first strictifying and then applying a classical infinite loop space machine to classifying spaces, this allows classical infinite 
loop space theory to construct morphisms between spectra from strong symmetric monoidal functors between symmetric monoidal categories.   Our theory will allow us to do the same thing equivariantly, starting from genuine symmetric monoidal $G$-categories, but we must first define what those are.

A pseudoalgebra over $\sP$ is a (small) symmetric  monoidal category.\footnote{This is true up to 
minor quibbles explained in \S\ref{genuine}}  This suggests the following new definition. We shall be more precise in \S\ref{genuine}. 

\begin{defn}\label{Gsymm} A {\em (genuine) symmetric monoidal $G$-category} is a $\sP_G$-pseudo\-algebra.  A strong symmetric monoidal functor of symmetric monoidal $G$-categories is a pseudomorphism of $\sP_G$-algebras.  A transformation between strong symmetric monoidal 
functors is a $\sP_G$-transformation.
\end{defn}   

Henceforward, when we say ``symmetric monoidal $G$-category" we always mean ``genuine." When we talk about \emph{naive} symmetric monoidal $G$-categories, we will always explicitly say ``naive."  The same convention applies to permutative $G$-categories.   As we  explain in \S\ref{genuine}, there is a functor that sends naive permutative $G$-categories to naively equivalent genuine permutative $
G$-categories and sends naive symmetric monoidal $G$-categories to naively equivalent 
genuine symmetric monoidal $G$-categories.  The functor applies to nonequivariant permutative and symmetric monoidal categories, viewed as $G$-categories with trivial $G$-action.   This gives a plentitude of examples.  

We discuss the philosophy behind \autoref{Gsymm} in \S \ref{genuine}, where we also indicate relevant categorical questions that have been addressed by Rubin \cite{Rubin1, Rubin2} in work complementary to ours.  He works concretely in the equivariant context of $N_{\infty}$ $G$-operads pioneered by Blumberg and Hill \cite{BH} and developed further by Rubin and others \cite{BP, GW, Rubin1}, and he compares our symmetric monoidal $G$-categories  with the analogous but definitionally disparate 
context of $G$-symmetric monoidal categories of  Hill and Hopkins \cite{HH}.  We shall say a bit more about his work in \S\ref{genuine}. 

It is not obvious that (genuine) symmetric monoidal $G$-categories are equivalent to (genuine) permutative $G$-categories, but \autoref{ConjOut1} shows that they are. 

\begin{cor}\label{strictcat}   The inclusion of permutative $G$-categories in symmetric monoidal $G$-categories has a left 2-adjoint strictification  $2$-functor.  For a symmetric monoidal $G$-category $\sX$, the unit $\sX \rtarr  \mathbf{St}\sX$  of the adjunction is an equivalence of symmetric monoidal $G$-categories.
\end{cor}

Combined with the results of \cite[\S4.5]{GM3}, this gives the following conclusion.

\begin{thm}
There is a functor $\mathbb{K}_G$ from symmetric monoidal $G$-categories to $\OM$-$G$-spectra such that $\OM^\infty \mathbb{K}_G(\sA)$  is an equivariant group completion of the classifying $G$-space $B\sA$.
\end{thm}

Thus $\mathbb{K}_G$ takes $\sP_G$-pseudoalgebras and $\sP_G$-pseudomorphism to  genuine $G$-spectra and maps of $G$-spectra; it even takes algebra 2-cells between 
$\sP_G$-pseudo\-morphisms to homotopies between maps of $G$-spectra (\autoref{hom}).
The proofs give explicit constructions.
Even nonequivariantly, this is a generalization of previous published work, although this
specialization has long been understood as folklore. At least on a formal level, this, coupled with \cite{ GM3, MMO}, completes the development of additive equivariant infinite loop space theory.   

\subsection*{Acknowledgements}   The essential ideas in this paper come from the beautiful categorical  papers by Power \cite{Power} and Lack \cite{Lack} and from earlier categorical
work of Kelly and Street, for example in \cite{BKP, StMon}.  This paper is a testament to the power of ideas in the categorical literature.  
We owe an enormous debt of gratitude to Steve Lack, John Power, Nick Gurski, and Mike Shulman for all of  their help.  We also thank Jonathan Rubin for the nice observation recorded in \S\ref{AppRubin}, which helps justify our framework of internal rather than just enriched categories. Finally, we thank the diligent referee for a careful reading.
 
\section{Categorical preliminaries}\label{prelim}

\subsection{Internal categories}\label{Pre1}

We need some elementary category theory to nail down relevant details about our general context.
In part to do equivariant work without working equivariantly, we work in a context of 
internal $\sV$-categories, where $\sV$ is any category with all finite limits. 
Some obvious examples are the category $\mathbf{Set}$ of sets, the category $\bf{Cat}$ of (small) categories, 
and the category $\sU$ of spaces,\footnote{As usual, spares are taken to be compactly generated and weak Hausdorff.}
but there are many others.  All examples come with based and equviariant variants, and the 
latter are of special interest to us.

\begin{rem} The category $\sV$ has a terminal object $\ast$, namely the product of the empty set of objects. 
A based object in $\sV$ is an object $V$ with a choice of morphism $v_0\colon \ast\rtarr V$. 
A based map $(V,v_0)\rtarr (W,w_0)$ is a morphism $V\rtarr W$ that is compatible with the choices of basepoint, and $\sV_*$ denotes
the category of based objects and based morphisms. Finite limits in $\sV_*$ are finite limits in $\sV$ with the induced map from $\ast$ given by the universal property.
\end{rem}

\begin{rem}\label{DefnGV} Let $G$ be a discrete group.  A $G$-object $V$ in $\sV$ has an action of $G$ 
given by automorphisms $g\colon V\rtarr V$ satisfying the evident unit and composition axioms. 
A $G$-map is a morphism $V\rtarr W$ that is compatible with given group actions, and $G\sV$ denotes the 
category of $G$-objects and $G$-maps. Finite limits in $G\sV$ are finite limits in $\sV$ with the induced action by $G$.
\end{rem}

We understand $\sV$-categories to mean internal $\sV$-categories and we recall the definition. 

\begin{defn}\label{Vcatfun} A \emph{$\sV$-category} $\sC$ consists of
objects $\mathbf{Ob}\,\sC$ and $\mathbf{Mor}\,\sC$ of $\sV$ with  source, target, identity, and 
composition maps $S$, $T$, $I$, and $C$ in $\sV$ that satisfy the axioms
of a category.  A \emph{$\sV$-functor} $f\colon \sC\rtarr \sC'$ is 
given by object and morphism maps in $\sV$ that commute with $S$, $T$, $I$, and $C$.
We write $\mathbf{Cat}(\sV)$ for the category of $\sV$-categories and $\sV$-functors.
\end{defn}

By contrast, a small category $\sD$ enriched in $\sV$ is given by a set of objects and an object $\sD(c,d)$ 
of $\sV$ for each pair $(c,d)$ of objects of $\sD$, 
with composition given by maps in $\sV$ and identities 
given by maps $\ast \to \sD(c,c)$ in $\sV$.

\begin{warn} In the categorical literature, $\sV$-categories usually refer to the enriched rather
than the internal notion.  In the unbased case, we can use the functor $\bV\colon \bf{Set} \rtarr \sV$ of \autoref{embed} below  to view categories enriched over $(\sV,\times)$ as  special cases of internal ones.
\end{warn}

\begin{exmp}  A $2$-category is a category enriched in $\bf{Cat}$, and its enriched functors
are called $2$-functors. A category internal to $\bf{Cat}$ is a double category, and the internal functors are double functors.
\end{exmp}

\begin{rem}\label{basedcat}
Since $\sV$ has a terminal object,  so does $\bf{Cat}(\sV)$.   It is easily checked that the categories $\bf{Cat}(\sV)_*$ and $\bf{Cat}(\sV_*)$ are canonically isomorphic. 
We shall use the notation $\bf{Cat}(\sV_*)$.
\end{rem}

\begin{rem} A $G\sV$-category is a category internal to $G\sV$.
Thus $G$ acts on both the object of objects and  the object of morphisms via morphisms in $\sV$. One can easily check that $\bf{Cat}(G\sV)$ is canonically isomorphic to $G\bf{Cat}(\sV)$. We are especially
interested in $G\sU$.
\end{rem}

\begin{rem}\label{Vembed}  One reason to require internal $\sV$-categories rather than just enriched ones 
is that it allows us to define an inclusion $i\colon \sV \to \mathbf{Cat}(\sV)$. We simply view an object $X$ 
of $\sV$ as a discrete $\sV$-category $iX$ with $\mathbf{Ob}(iX)=\mathbf{Mor}(iX)=X$, and $S$, $T$, and $I$ all identity maps, and $C$ the canonical isomorphism $X\times _X X \cong X$. It is straightforward to check that $i$ is full and faithful and is left adjoint to the object functor.  Thus
\[ \bf{Cat}(\sV)(iX,\sA) \iso \sV(X,\bf{Ob}\sA). \]
We often omit $i$ from the notation, regarding $\sV$ as a full subcategory of $\mathbf{Cat}(\sV)$. 
\end{rem}

Along with the $\sV$-categories and $\sV$-functors of \autoref{Vcatfun}, we need $\sV$-natural transformations, which we abbreviate to $\sV$-transformations.

\begin{defn}\label{Vtrans} A {\em $\sV$-transformation} $\al\colon f\Longrightarrow g$, where $f$ and $g$ are $\sV$-functors 
$\sA\rtarr \sB$, 
is a map $\al\colon \mathbf{Ob}\, \sA\rtarr \mathbf{Mor}\, \sB$ in $\sV$ such that the following two diagrams commute.
\begin{equation}\label{vtrans1}
\xymatrix{
& \mathbf{Mor}\,\sB \ar[d]^{(S,T)}\\
\mathbf{Ob}\,\sA \ar[r]_-{(f,g)} \ar[ur]^{\al} & \mathbf{Ob}\,\sB\times \mathbf{Ob}\,\sB\\} \end{equation}
\begin{equation}\label{vtrans2} \xymatrix{
\mathbf{Ob}\, \sA\times \mathbf{Mor}\,\sA \ar[d]_{\al\times f}
& \mathbf{Mor}\,\sA \ar[l]_-{(T,\Id)} \ar[r]^-{(\Id,S)}
& \mathbf{Mor}\,\sA\times \mathbf{Ob}\,\sA\ar[d]^{g\times \al}\\
\mathbf{Mor}\,\sB\times_{\mathbf{Ob}\sB} \mathbf{Mor}\,\sB \ar[r]_-{C} & \mathbf{Mor}\,\sB 
& \mathbf{Mor}\,\sB\times_{\mathbf{Ob}\sB} \mathbf{Mor}\,\sB \ar[l]^-{C}\\} \end{equation}
Note that the right down and left down composites do indeed land in the pullback, since $S\com \al\com T = f\com T = T\com f$ and $T\com \al\com S = g\com S = S\com g$. 

The vertical composite $\be\ast \al$ of $\al\colon f\Longrightarrow g$ and 
$\be\colon g\Longrightarrow h$ is the composite
\[\xymatrix@1{
\bf{Ob}\sA \ar[r]^-{(\be, \al)} 
& \bf{Mor}\sB\times_{\bf{Ob}\sB} \bf{Mor}\sB \ar[r]^-C & \bf{Mor}\sB.\\} \]
The identity $\sV$-transformation $\id\colon f \Longrightarrow f$ is given by 
\[f\circ I = I \circ f \colon \bf{Ob}\sA \longrightarrow \bf{Mor}\sB .\]
We say that $\al\colon f \Longrightarrow g$ is an isomorphism, or $\al$ is invertible, if there is a $\sV$-transformation 
$\al^{-1}\colon g\Longrightarrow f$ such that $\al\ast\al^{-1} = \id$ and $\al^{-1}\ast\al = \id$. As in $\bf{Set}$, the condition in \autoref{vtrans2} 
for $\alpha^{-1}$ follows from that for $\alpha$.

The horizontal composite $\beta\com \alpha$ of $\al $ and $\beta$, as in the diagram
\[\xymatrix{
\sA \rtwocell^f_{f'}{\al} & \sB \rtwocell^g_{g'}{\beta} & \sC,
}
\]
is given by the common composite in the commutative diagram
\[ \xymatrix{
\mathbf{Ob}\, \sB\times \mathbf{Mor}\,\sB \ar[d]_{\be\times g}
& \mathbf{Ob}\,\sA \ar[l]_-{(f',\al)} \ar[r]^-{(\al,f)}
& \mathbf{Mor}\,\sB\times \mathbf{Ob}\,\sB\ar[d]^{g'\times \be}\\
\mathbf{Mor}\,\sC\times_{\mathbf{Ob}\sC} \mathbf{Mor}\,\sC \ar[r]_-{C} & \mathbf{Mor}\,\sC 
& \mathbf{Mor}\,\sC\times_{\mathbf{Ob}\sC} \mathbf{Mor}\,\sC \ar[l]^-{C}\\}
\]
In particular, using the same notation as above, the whiskering $\beta \com f$ is given by the composite
\[\xymatrix{
\bf{Ob}\,\sA \ar[r]^{f} & \bf{Ob}\,\sB \ar[r]^{\be} & \bf{Mor}\,\sC,}\]
and similarly, the whiskering $g \com \al$ is given by the composite
\[\xymatrix{
\bf{Ob}\,\sA \ar[r]^{\al} & \bf{Mor}\,\sB \ar[r]^{g} & \bf{Mor}\,\sC.
}
\]
\end{defn} 

\begin{notn}
Let $\sV$ be a category with finite limits. Then the collection of $\sV$-categories, $\sV$-functors, and  $\sV$-transformations forms a 2-category, which we will also denote by  $\mathbf{Cat}(\sV)$, updating the notation of \autoref{Vcatfun}.  In particular, we have the updated  notations $\bf{Cat}(\sV_\ast)$ and $\bf{Cat}(G\sV)$ for the based and equivariant variants
viewed as $2$-categories.

\end{notn}

 \subsection{Chaotic categories}\label{sec:chaotic}
 
We recall the definition of chaotic (or indiscrete) category in the general context of internal categories. 
 
 \begin{defn} A $\sV$-category $\sC$ is said to be {\em chaotic} (or {\em indiscrete}) if the map
\[ \mathbf{Mor}(\sC) \xrightarrow{(S,T)} \bf{Ob}(\sC) \times \bf{Ob}(\sC)\]
is an isomorphism in $\sV$.
\end{defn}

Chaotic $\sV$-categories, despite their simplicity, are important since they lead to natural 
constructions of operads in $\sV$.   An ordinary category $\sA$ is chaotic if each 
$\sA(x,y)$ is a point. For a set $X$ there is a canonical chaotic category $\mathcal{E} X$ with object set $X$.  This is related to other constructions in \cite[\S1]{GMM}. We saw in \autoref{Vembed} that the object functor $\bf{Ob}\colon \bf{Cat}(\sV)\rtarr \sV$ has a left adjoint inclusion functor $i$; the chaotic category functor is right adjoint to $\bf{Ob}$, as we show in \autoref{chaadj} below.  To generalize to $\sV$-categories, we start with the construction of $\mathcal{E} X$.

\begin{defn}\label{chaotic0}  Let $X$ be an object of $\sV$. The {\em chaotic $\sV$-category} $\mathcal{E} X$ has
$\mathbf{Ob}\,\mathcal{E} X = X$ and $\mathbf{Mor}\,\mathcal{E} X= X\times X$. The maps
$S$, $T$, and $I$ are the projections $\pi_2$, $\pi_1$, and the diagonal $\DE$ respectively, and the map $C$ is
\[ \id\times \epz\times \id\colon (X\times X)\times_X(X\times X) \iso X\times X\times X \rtarr X\times X,\]
where $\epz\colon X\rtarr \ast$; that is, $C$ is projection onto the first and third coordinates.  
\end{defn} 

\begin{rem}\label{contract} When $\sV=\bf{Set}$, 
 every object of $\mathcal{E} X$ is initial and terminal, so that $\ast$ is isomorphic to a skeleton of $\mathcal{E} X$. Therefore $B\mathcal{E}X$ is contractible. This also applies when $\sV$ is the category of spaces. 
\end{rem}

\begin{lem}\label{chaadj} The chaotic $\sV$-category functor $\mathcal{E}\colon \sV\rtarr \mathbf{Cat}(\sV)$ is right 
adjoint to the object functor $\mathbf{Ob}$, so that there is a natural isomorphism of sets
\[ \sV(\mathbf{Ob}\, \sA, X) \iso \mathbf{Cat}(\sV)\,(\sA,\mathcal{E} X). \]
Moreover, for any two 
$\sV$-functors $E,F\colon \sA \rtarr \mathcal{E} X$, there exists a unique $\sV$-transformation $\al\colon E\rtarr F$,
necessarily a $\sV$-isomorphism.
\end{lem} 
\begin{proof} The $\sV$-functor $F\colon \sA\rtarr \mathcal{E} X$ corresponding to
a map $f\colon \mathbf{Ob}\,\sA \rtarr X$ in $\sV$ is given by $f$ on objects and by 
\[ \xymatrix@1{ \mathbf{Mor}\,\sA\ar[r]^-{(T,S)} 
& \mathbf{Ob}\,\sA\times \mathbf{Ob}\,\sA \ar[r]^-{f\times f} & X\times X \\} \]
on morphisms.  Thus $\mathbf{Ob}\, F =f$ by definition, and a little diagram chase shows that
$F$ is the only $\sV$-functor with object map $f$. 

Given $\sV$-functors $E$ and $F$ and a $\sV$-transformation $\al\colon E\Longrightarrow F$, the condition in \autoref{vtrans1} forces $\al=(F,E)$. Again, a small diagram chase shows that $\al$ so defined is indeed a $\sV$-transformation.
\end{proof}

The following result is a reinterpretation of the second statement of \autoref{chaadj}.

\begin{cor}\label{chaotic}
The  \emph{category} of $\sV$-functors and $\sV$-natural transformations from $\sA$ to $\mathcal{E} X$ is isomorphic to the chaotic category on the \emph{set} of $\sV$-maps from $\bf{Ob} \sA$ to $X$.
\end{cor}

Note that the counit $\bf{Ob}\com \mathcal{E}\rtarr \Id$ of the adjunction is the identity. 

\begin{lem} The unit map $\sA\rtarr \mathbf{\mathcal{E}}(\mathbf{Ob}\sA)$ of the adjunction is an isomorphism if and only if the $\sV$-category $\sA$ is chaotic.
\end{lem}

As a right adjoint, the chaotic category functor preserves products and other limits and
therefore preserves all structures defined in terms of those operations. We can view it 
as an especially elementary form of categorification. 

\subsection{The embedding of $\bf{Set}$ in $\sV$}\label{embed}

Many operads and other constructions are first defined in the category $\bf{Set}$.  In the unbased case, assuming that $\sV$ has coproducts in addition to finite limits, we can use the following definition to lift such constructions to $\sV$. 

\begin{defn}\label{bUbV}   Define $\bV\colon \bf{Set} \rtarr \sV$ to be the functor that sends a set 
$S$ to $\coprod_{s\in S} \ast$, the coproduct of copies of the terminal object $\ast$ indexed on $S$. 
It has a right adjoint $\bU\colon \sV\rtarr \mathbf{Set}$ specified by letting $\bU X = \sV(\ast,X)$.  Thus
\begin{equation}\label{trivadj1}
  \sV(\bV S,X) \iso \mathbf{Set}(S,\bU X).
\end{equation}
\end{defn}

\begin{rem}\label{set}
In all of the unbased examples of interest, the unit map $\Id \rtarr \bU\bV$ of the adjunction is an isomorphism. This expresses the intuition that a map from a point into a disjoint union of points is the same as a choice of one of the points. It ensures that $\bV$ is  a full and faithful functor.  Henceforward, in the unbased case, we assume this and thus regard $\mathbf{Set}$ as a subcategory of $\sV$, omitting $\bV$ from the notation.
\end{rem}

\begin{rem}
When the unit $\Id \rtarr \bU\bV$ of the adjunction between $\bf{Sets}$ and $\sV$ is an isomorphism, the adjunction lifts to an adjunction between $\bf{Sets}_*$ and $\sV_*$. Indeed, we define $\bV$ of a set $S$ with basepoint $s_0$ to be the based object
\[\ast \cong \bV(\ast) \xrightarrow{\bV(s_0)} \bV S,\]
and similarly, we define $\bU$ of a based object $(X,x_0)$ in $\sV$ to be
\[\ast \cong \bU\bV(\ast)\cong \bU(\ast)\xrightarrow{\bU(x_0)} \bU X.\]
The unit and the counit of the original adjunction then become based maps, giving the desired adjunction.
\end{rem}

\begin{defn}\label{GbUbV} The adjunction between $\bf{Sets}$ and $\sV$ also lifts to the equivariant setting in the following way. 
Define $\bV\colon G\bf{Set} \rtarr G\sV$ to be the functor that sends a $G$-set $S$ to the object $\bV S$ in $\sV$ with the action of $G$ induced by the functoriality of $\bV$ applied to the maps of sets $g\colon S\rtarr S$ for $g\in G$.  Thinking of the action by $G$ on an object $X$ of $G\sV$ as given by a map $\bV G\times X\rtarr X$ in $\sV$ and applying $\bU$, we obtain an action of 
$G$ on $\bU X$.  This gives a forgetful functor $\bU \colon G\sV\rtarr G\bf{Set}$ that is right adjoint 
to $\bV$.  Thus
\begin{equation}\label{trivadj2}
  G\sV(\bV S,X) \iso G\mathbf{Set}(S,\bU X).
\end{equation}
\end{defn}

The following remark applies equally well in the nonequivariant and equivariant contexts.

\begin{rem}\label{finitelim}  As a left adjoint, $\bV$ preserves colimits.  To ensure that $\bV$ preserves operads and other structure in $\bf{Set}$, we assume henceforward that $\bV$ also preserves finite limits.  As we explain in the brief \autoref{AppRubin}, which was provided to us by Jonathan Rubin, 
this is a very mild assumption that holds in all of our unbased examples. The assumption ensures that the adjunction $(\bV,\bU)$, when applied to objects and morphisms, induces an adjunction
\begin{equation}\label{trivadj3} 
\mathbf{Cat}(\sV)(\bV \sA,\sB) \iso \mathbf{Cat}(\sA,\bU \sB),
\end{equation}
where $\sA$ is a category and $\sB$ is a $\sV$-category. The functor $\bV\colon \mathbf{Cat}\rtarr \mathbf{Cat}(\sV)$ is again full and faithful, and we regard $\mathbf{Cat}$ as a subcategory of $\mathbf{Cat}(\sV)$,  omitting $\bV$ from the notation. 
\end{rem}

We end this section by noting that using the functor $\bV$ and assuming that $\sV$ is cartesian closed, one can see that $\sV$-transformations can be thought of as analogues of homotopies. 
Let $\sI$ be the category with objects $[0]$ and $[1]$ and a unique 
non-identity morphism $I\colon [0]\rtarr [1]$, and consider it as a $\sV$-category via the functor $\bV$. For $\sV$-functors $f,g\colon \sA\rtarr \sB$, 
there is a bijection between $\sV$-transformations from $f$ to $g$ and  $\sV$-functors 
$h\colon \sA\times \sI \rtarr \sB$ that restrict to $f$ on $\sA\times [0]$ and to $g$ on $\sA\times [1]$.
Indeed, given $\al\colon \mathbf{Ob}\, \sA\rtarr \mathbf{Mor}\, \sB$, we define $h\colon 
\sA\times \sI \rtarr \sB$ on objects as
\[\mathbf{Ob}\, (\sA \times \sI)=\bf{Ob}\sA \times \coprod_{\{[0],[1]\}} \! \!* \cong \coprod_{\{[0],[1]\} } \mathbf{Ob}\,\sA \xrightarrow{\coprod \al} \coprod_{\{[0],[1]\} }\mathbf{Mor}\,\sB \xrightarrow{S,T} \mathbf{Ob}\,\sB.\]
On morphisms, $h$ is given by the $\sV$-functor
\[\mathbf{Mor}\,(\sA \times \sI)=\bf{Mor}\sA \times \coprod_{\{\id_0,\id_1,I\}} \!\!\! \!*  \cong \coprod_{\{\id_0, \id_1, I\}} \mathbf{Mor}\,\sA \rtarr \mathbf{Mor}\,\sB\]
specified on the three components of the coproduct by $f$, $g$ and the common composite in (\ref{vtrans2}), respectively. We leave it to the reader to check that this assignment is a bijection. 

\begin{rem}\label{hom} Taking $\sV = G\sU$, taking $\sO$ to be an $E_{\infty}$ $G$-operad in $\bf{Cat}(G\sU)$, and using that the classifying space functor $B$ preserves products and takes $\sI$ to the unit interval, we can use our infinite loop space machinery \cite{GMMO, MMO}, in particular \cite[Proposition 6.16]{GMMO}, to transport  
$G\sU$-transformations between strict maps of $\sO$-algebras to homotopies between maps of $G$-spectra. 
\end{rem}

\section{Pseudoalgebras over operads and $2$-monads}\label{pseudo}

\subsection{Pseudoalgebras over $2$-Monads}\label{Monpseudo}

\begin{defn}  A {\em $2$-monad} on a $2$-category $\sK$ is a $\bf{Cat}$-enriched monad in $\sK$.  
Precisely, it is a
$2$-functor $\bT\colon \sK\rtarr \sK$ together with $2$-natural transformations $\io\colon I\rtarr \bT$ 
and $\mu\colon \bT\bT\rtarr \bT$ satisfying the evident unit and associativity laws:
the following diagrams of $2$-natural transformations must commute.
\[ \xymatrix{
\bT \ar[r]^-{\io \bT} \ar@{=}[dr]& \bT^2 \ar[d]^{\mu} & \bT \ar[l]_-{\bT\io} \ar@{=}[dl]\\
& \bT & \\}
\ \ \ 
\xymatrix{ 
\bT^3 \ar[r]^{\mu \bT} \ar[d]_{\bT \mu} & \bT^2 \ar[d]^{\mu} \\
\bT^2 \ar[r]_-{\mu} & \bT\\} \]
\end{defn}

\begin{defn}\label{Tpseudo} A {\em (strict) $\bT$-algebra} $(X,\tha)$ is an object $X$ of $\sK$ together
with an action $1$-cell $\tha \colon \bT X \rtarr X$ such that the following diagrams 
commute. 
\[ \xymatrix{
X \ar[d]_-{\io_X } \ar@{=}[dr]\\
 \bT X \ar[r]_-{\tha} & X  \\}
\qquad
\xymatrix{ 
\bT^2 X \ar[r]^{ \bT \tha} \ar[d]_{\mu_X} & \bT X \ar[d]^{\tha} \\
\bT X \ar[r]_-{\tha} & X.\\} \]
In particular, $\bT X$ is a $\bT$-algebra with action map $\mu$ for any $X\in \sK$. 

A $\bT$-pseudoalgebra 
$(X,\tha,\varphi,\upsilon)$ requires the same two diagrams to commute up to invertible $2$-cells
\[ \upsilon\colon \id \Longrightarrow \tha\com \io_X\ \ \text{and}\ \ \varphi\colon \tha\com T\tha 
\Longrightarrow \tha\com \mu_X, \]
satisfying three coherence axioms (\cite[2.4]{Power}).
One defines lax $\bT$-algebras similarly, but not requiring $\upsilon$ and $\varphi$ to be invertible. We shall not consider them. 

A $\bT$-pseudoalgebra is {\em normal} if the first diagram commutes, so that $\upsilon$ is the identity.  We restrict attention to {\em normal} pseudoalgebras henceforward.\footnote{This is dictated by our preferred definitions when we turn to operads.  See \autoref{unitax}.}   With this restriction, the first two coherence axioms translate to requiring that the whiskerings $\varphi\com \io_{\bT X}$ and $\varphi\com \bT\io_X$ are both the identity transformation $\tha \Longrightarrow \tha$. The remaining coherence axiom requires the  equality of diagrams
\[ \xymatrixcolsep{.5cm}\xymatrixrowsep{.5cm}\xymatrix{
\bT^3 X\ar[rr]^-{\bT^2 \tha} \ar[dd]_-\mu &&  \bT^2 X\ar[dd]_-\mu \ar[dr]^-{\bT\tha}  & \\
&&  \drtwocell<\omit>{<0> \, \varphi} &  \bT X  \ar[dd]^-\tha\\
 \bT^2 X \ar[rr]^-{\bT \tha} \ar[dr]_-\mu & \drtwocell<\omit>{<0> \, \varphi}  &   \bT X \ar[dr]^-\tha  &\\
&  \bT X \ar[rr]_-\tha && X 
}  \ \ \ \ \ \ \  \xymatrix{&\\&\\=\\&}   \xymatrix{
\bT^3 X\ar[rr]^-{\bT^2 \tha} \ar[dd]_-\mu  \ar[dr]_-{\bT \mu} &  \drtwocell<\omit>{<0> \,\, \bT\varphi} & \bT^2 X \ar[dr]^-{\bT \tha}  &\\
& \bT^2 X  \ddrrtwocell<\omit>{<0> \, \varphi}   \ar[dd]^-\mu \ar[rr]^-{\bT \tha} &&  \bT X \ar[dd]^-\tha \\
\bT^2 X \ar[dr]_-\mu &&&\\
& \bT X\ar[rr]_-\tha && X
}\]
\end{defn} 

\begin{defn}\label{Tpseudo2} A {\em $\bT$-pseudomorphism} $(f,\ze)\colon (X,\tha,\varphi)\rtarr (Y,\xi,\psi)$ of $\bT$-pseudo\-algebras is given by a $1$-cell $f\colon  X\rtarr Y$ and an invertible $2$-cell 
$\ze\colon \xi \com \bT f\Longrightarrow f\com \tha$.
\[ \xymatrix{
\bT X \ar[r]^-{\bT f} \ar[d]_{\tha} \drtwocell<\omit>{<0> \, \ze} & \bT Y \ar[d]^{\xi}\\
X \ar[r]_-{f} & Y\\} \]
satisfying two coherence axioms (\cite[2.5]{Power}).
If $\ze$ is the identity, $f$ is said to be a strict $\bT$-map.  One defines lax $\bT$-maps by not requiring $\ze$ to be invertible, but we shall not consider those. 

Restricting $X$ and $Y$ to be normal, we require the whiskering $\ze \com \io_X$ to be the 
identity transformation $f \Longrightarrow f$.  This makes sense since the naturality of $\io$ and 
the normality equalities $\tha\com \io_X = \id_X$ and $\xi\com \io_Y = \id_Y$ show that the domain and target of $\ze\com \io_X$ are both $f$. 
There is then only one remaining coherence axiom.  It requires the equality of diagrams
\[ 
   \xymatrixcolsep{.5cm}\xymatrixrowsep{.5cm}\xymatrix{
\bT^2 X\ar[rr]^-{\bT^2 f} \ar[dd]_-\mu  \ar[dr]_-{\bT \tha} &  \drtwocell<\omit>{<0> \,\, \bT\ze} & \bT^2 Y \ar[dr]^-{\bT \xi}  &\\
 \drtwocell<\omit>{<0> \,\, \varphi}  & \bT X  \ddrrtwocell<\omit>{<0> \, \ze}   \ar[dd]^-\tha \ar[rr]_-{\bT f} &&  \bT Y \ar[dd]^-\xi \\
\bT X \ar[dr]_-\tha &&&\\
& X\ar[rr]_-f && Y
}  \ \ \ \ \ \ \  \xymatrix{&\\&\\=\\&}
\xymatrix{
\bT^2 X\ar[rr]^-{\bT^2 f} \ar[dd]_-\mu &&  \bT^2 Y\ar[dd]_-\mu \ar[dr]^-{\bT\xi}  & \\
&&  \drtwocell<\omit>{<0> \, \psi} &  \bT Y  \ar[dd]^-\xi\\
 \bT X \ar[rr]^-{\bT f} \ar[dr]_-\tha & \drtwocell<\omit>{<0> \, \ze}  &   \bT Y \ar[dr]^-\xi  &\\
&  X \ar[rr]_-f&& Y 
}\]
\end{defn}

\begin{defn}\label{Tpseudo3}  An {\em algebra 2-cell}
$\la \colon (f,\ze)\Longrightarrow (g,\ka)$ 
is given by a $2$-cell $\la\colon f\Longrightarrow g$ in $\sK$, not necessarily invertible, such that
\[ \xymatrixrowsep{.5cm}\xymatrix{
\bT X  \rtwocell^{\bT f}_{\bT g}{<0> \, \bT\la}  \ar[dd]_{\tha}  \ddrtwocell<\omit>{<.5> \, \ka}& \bT Y \ar[dd]^{\xi} 
& & \bT X \ar[dd]_{\tha}   \ar[r]^{\bT f}    \ddrtwocell<\omit>{<-.5> \,\, \ze} & \bT Y \ar[dd]^-{\xi} \\
&&=\\
X \ar[r]_g & Y & & X \rtwocell^{f}_{g}{<0> \, \la}  &  Y \\}
\]
\end{defn}

With these definitions, we have the three $2$-categories $\tpa$, $\etp$, and $\ets$ promised in the introduction.

\subsection{The $2$-monads associated to operads}\label{OpMon} 

To construct a monad from an operad, we must assume that $\sV$ and  $\bf{Cat}(\sV)$ have colimits in addition to having finite limits.  The construction of the monad associated to an operad requires equivariance and base object identifications, which are examples of colimits.  Since colimits of categories are often notoriously ill-behaved,  we offer a philosophical comment on how we use the $2$-monads associated to operads in topology.

\begin{rem}
We are interested in $\sO$-$G$-categories $\sX$ and their classifying $G$-spaces $X = B\sX$.   No
monads need play any role in the statements of the theorems we are proving about them, but we are using $2$-monads on categories of $G$-categories for the proofs.  With some exceptions, we neither know nor care about any commutation properties of $B$ relating these $2$-monads to monads on categories of $G$-spaces.   Such relations would be suspect since we cannot expect the relevant colimits to commute with $B$.   That is, we are using $2$-monads purely formally to obtain information about the underlying categories of $\sO$-$G$-algebras. 
\end{rem} 

Operads are defined in any symmetric monoidal category and in particular in any cartesian monoidal category.  An operad $\sO$ in $\mathbf{Cat}(\sV)$ consists of $\sV$-categories 
$\sO(j)$ for $j\geq 0$ with right actions of the symmetric groups $\SI_j$,  a unit $\sV$-functor  $1\colon \ast \rtarr \sO(1)$, 
where $\ast$ is the trivial $\sV$-category, and structure $\sV$-functors 
\[ \ga\colon \sO(k)\times \sO(j_1) \times \cdots \times \sO(j_k) \rtarr \sO(j_1 + \cdots + j_k)  \]
that are equivariant, unital, and associative in the sense that is prescribed in \cite[Definition 1.1]{MayGeo}.  

\begin{ass}\label{reduced}
We assume throughout that operads $\sO$ are taken to be reduced operads in $\mathbf{Cat}(\sV)$.    
Reduced means that $\sO(0)$ is the terminal object $\ast$, so that an $\sO$-algebra $\sA$ has a base object $0$, 
namely the image of $\ast$ under the action. We write $0$ for the identity $\sV$-functor $\ast\rtarr \sO(0)$.
\end{ass} 

For the most useful contexts, we must also assume that $\sO$ is $\SI$-free, meaning that the symmetric group $\SI_j$ acts freely on the $j$th object $\sO(j)$ for all $j$, but we do not restrict to $\SI$-free operads in this paper.

We shall be especially interested in chaotic operads.

\begin{defn}  An operad $\sO$ in $\mathbf{Cat}(\sV)$ is {\em chaotic} if each of its $\sV$-categories $\sO(n)$ is 
chaotic.
\end{defn}

We will shortly define strict algebras and pseudoalgebras over an operad in $\bf{Cat}(\sV)$. For an operad $\sO$ in any symmetric monoidal category $(\sW, \otimes)$, we have an isomorphism of categories between (strict) $\sO$-algebras and $\bO_+$-algebras, where $\bO_+$ is the monad on $\sW$ that is constructed from 
$\sO$ by defining
\begin{equation}\label{OMonad}  \bO_+ X = \coprod_{n\geq 0} \sO(n) \otimes_{\SI_n} X^{\otimes n}.
\end{equation}
Note that $\SI_n$ acts on the right of $\sO(n)$ and on the left of $X^{\otimes n}$.  Intuitively, we are identifying 
$a\rh\otimes x$ with $a\otimes \rh x$ for $\si\in \SI_n$ and elements $a\in \sO(n)$ and $x\in X^{\otimes n}$. 

As explained in \cite[\S4]{Rant1}, if $\sW$ is cartesian monoidal and $\sO$ is reduced,
there is a monad $\bO$ on $\sW_*$ whose (strict) algebras are the same as those of $\bO_+$. The difference is that $\bO_+$-algebras acquire  base objects via their actions, whereas $\bO$-algebras have preassigned base objects that must agree with those assigned by their actions; $\bO$ is constructed  from $\bO_+$ using base object identifications. We can adjoin disjoint base objects by taking 
$X_+ = X\amalg \ast$, and then $\bO_+(X) = \bO(X_+)$.   In all topological applications, the monad $\bO$ is of considerably greater interest than the monad $\bO_+$, and we shall restrict attention to it.

We need a preliminary definition to define $\bO$ in our context. 

\begin{defn}\label{degen}
Let $\sO$ be an operad in $\bf{Cat}(\sV)$ and let $\sA$ be a based $\sV$-category. In line with \autoref{reduced}, let $0$ denote the base object  of $\sA$.  Let $1\leq r\leq n$.  Define
$\si_r\colon \sO(n) \rtarr \sO(n-1)$ to be the composite $\sV$-functor 
\begin{equation}\label{degen1} 
\xymatrix{
   \sO(n) \cong \sO(n)\times \ast^n \ar[d]^{\id \times 1^{r-1}\times 0 \times 1^{n-r}}\\
  \sO(n)\times \sO(1)^{r-1}\times \sO(0) \times \sO(1)^{n-r} \ar[d]^{\ga}\\
 \sO(n-1).}
\end{equation}
Define $\si_r \colon \sA^{n-1}\rtarr \sA^n$ to be the insertion of  base object $\sV$-functor
\begin{equation}\label{degen2}
\si_r = \id^{r-1}\times 0 \times \id^{n-r}\colon \sA^{n-1}\rtarr \sA^n.
\end{equation}
\end{defn}

\begin{con}\label{OMon}  
Let $\sO$ be a (reduced) operad in $\bf{Cat}(\sV)$.  We construct a 2-monad $\bO$ in the $2$-category 
$\bf{Cat}(\sV_*)$ of based $\sV$-categories. Let $\LA$ be the subcategory of
injections and permutations in the category of finite based sets $\mathbf{n}$.  Then $\sO$ is a contravariant functor on $\LA$ via the 
symmetric group actions and the degeneracy functors $\si_r$.  For a based
$\sV$-category $\sA$, the powers $\sA^n$ give a covariant functor $\sA^{\bullet}$ on $\LA$ via permutations and the insertions of base object functors $\si_r$. Define
\begin{equation}\label{Omon} \bO(\sA) =  \sO\otimes_{\LA} \sA^{\bullet},
\end{equation}
 where $\otimes_{\LA}$ denotes a tensor product of functors, as defined in \cite[\S~IX.6]{CatWork}.
The unit $\io\colon \sA\rtarr \bO \sA$ is induced by the $\sV$-map 
$\ast \rtarr \sO(1)$ determined by $\id\in \sO(1)$ and the product $\mu\colon \bO^2\rtarr \bO$ 
is induced by the structural maps $\ga$ of the operad. 
\end{con} 

\subsection{Pseudoalgebras over operads}\label{Oppseudo}
We define pseudoalgebras over an operad $\sO$ in $\mathbf{Cat}(\sV)$, 
largely following Corner and Gurski \cite{CG}.\footnote{They only consider $\sV = \mathbf{Set}$, but the generalization is immediate.}  The definition can be extended to operads in any 2-category with products.

\begin{defn}\label{Opseudo} An {\em $\sO$-pseudoalgebra} $\sA = (\sA,\tha,\varphi)$ is a $\sV$-category $\sA$ together with action $\sV$-functors 
\[ \tha = \tha_n \colon \sO(n) \times \sA^n\rtarr \sA \]
and invertible composition $\sV$-transformations $\varphi = \varphi(n;m_1,\cdots,m_n)$ 
\begin{equation}\label{cohere}
 \xymatrix{ \sO(n)\times \left(\prod_{r} \sO(m_r)\times \sA^{m_{r}}\right) \ar[dd]_{\pi} \ar[rr]^-{\id\times (\prod_r\tha_{m_r})}  
\ddrrtwocell<\omit>{<0>   \varphi} 
& &
\sO(n)\times \sA^n \ar[dr]^{\tha_n }  &\\
& & &\sA\\
\sO(n)\times \left(\prod_r{\sO(m_r)}\right) \times \sA^{m} \ar[rr]_-{\ga\times \id} & &\sO(m)\times \sA^m. \ar[ur]_{\tha_m}\\}
\end{equation}

\noindent
Here $1\leq r\leq n$, $m=m_1 +\cdots + m_n$, and $\pi$ is the shuffle that moves the 
variables $\sO(m_r)$ to the left and identifies $\sA^{m_1}\times \cdots \times \sA^{m_n}$ with $\sA^m$.
These data must satisfy the following axioms.  When we say that an instance of (\ref{cohere}) commutes, we mean that the corresponding component of $\varphi$ is the identity.

\begin{ax}[Equivariance] The following diagram commutes for $\rh\in \SI_n$.
\[ \xymatrix{ 
\sO(n) \times \sA^{n} \ar[d]_-{\rh\times\id} \ar[r]^{\id\times \rh} & \sO(n) 
\times \sA^{n} \ar[d]^{\tha_n } \\
\sO(n) \times \sA^{n} \ar[r]_-{\tha_n } & \sA\\} \]
This means that the $\tha$ induce a map $\tha\colon \bO_+\sA \rtarr \sA$.
\end{ax}

\begin{ax} [Unit Object]\label{unitax}  For $1\leq r\leq n$, the following whiskering of an instance of the diagram \autoref{cohere}
commutes. That is, the whiskering of $\varphi$ along the composite of the  first map of \autoref{degen1} and an instance of $\pi^{-1}$ is the identity 2-cell, giving the following commutative diagram.
\[ \xymatrix{ 
\sO(n) \times \sA^{n-1} \ar[d]_-{\si_r\times\id} \ar[r]^{\id\times \si_r} & \sO(n) \times \sA^{n} \ar[d]^{\tha_n } \\
\sO(n-1) \times \sA^{n-1} \ar[r]_-{\tha_{n-1}} & \sA\\} \] 
This means that the $\tha$ induce a map $\tha\colon \bO \sA \rtarr \sA$.
\end{ax} 

\begin{ax}\label{opid}[Operadic Identity]
The following diagram commutes.
\[ \xymatrix{
\ast \times \sA \ar[dr]_{\iso} \ar[r]^-{1\times \id} & \sO(1)\times \sA \ar[d]^{\tha_{1}} \\
& \sA \\}\] 
\end{ax}

We require coherence axioms for the $\sV$-transformations $\varphi$. These are dictated by compatibility with the monadic axioms in \S\ref{Monpseudo} and we use those to abbreviate the statements of the operadic axioms. 

\begin{ax}\label{Pone} [Equivariance] When the diagram \autoref{cohere} is obtained from another such diagram by precomposing with a permutation, we require $\varphi$ to be the whiskering of $\varphi$ in the original diagram by the permutation. Precisely, given $\rho\in \Sigma_n$ and $\ta_r\in \Sigma_{m_r}$, there are equalities of whiskerings
\[\varphi(n;m_1,\dots,m_r)=\varphi(n;m_{\rho(1)},\dots,m_{\rho(n)})  \com (\rho \times \rho^{-1})\]
and
\[\varphi(n;m_1,\dots,m_r) =\varphi(n;m_1,\dots,m_r)  \com \left(\id \times \prod_r(\ta_r \times \ta_r^{-1})\right).\]  
This means that the $\varphi$ pass to orbits to define an invertible $2$-cell, 
which we also denote by $\varphi$,
in the diagram
\[ \xymatrix{
\bO^2_+\sA  \ar[r]^-{\bT \tha} \ar[d]_-{\mu} \drtwocell<\omit>{<0> \, \varphi}  &  \bO_+\sA  \ar[d]^{\tha}\\
\bO_+ \sA \ar[r]_{\tha} & \sA.}
\] 
\end{ax}

Using the unit object axiom, it follows that $\varphi$ then passes through base object identifications to define
an invertible $2$-cell, which we again call $\varphi$, in the diagram
\begin{equation}\label{Opvarphi}
 \xymatrix{
\bO^2\sA  \ar[r]^-{\bT \tha} \ar[d]_-{\mu} \drtwocell<\omit>{<0> \, \varphi}  &  \bO\sA  \ar[d]^{\tha}\\
\bO \sA \ar[r]_{\tha} & \sA.}
\end{equation}

\begin{ax}\label{Pthree} [Operadic Identity]  The whiskering of $\varphi(1;n)$ along 
\[  1\times \id\colon \sO(n) \times \sA^n \rtarr \sO(1) \times  \sO(n) \times \sA^n \]
is the identity, and the whiskering of $\varphi(n;1^n)$ along 
\[ \id\times (1\times \id)^n\colon \sO(n)\times \sA^n\rtarr \sO(n)\times (\sO(1)\times \sA)^n \]
is the identity.
\end{ax}

\begin{ax}\label{Pfour} [Operadic Composition] 
Writing $\mu = (\gamma \times \id)\com \pi$, $m=\sum_r m_r$, $p_r=\sum_s p_{rs}$, and $p=\sum_{r,s} p_{rs}$,
we require the following two pasting diagrams to be equal.
{\footnotesize{
\[ \xymatrix@C=1.2em{
\sO(n)\times \prod_r\big(\sO(m_r)\times \prod_s(\sO(p_{rs}) \times \sA^{p_{rs}})\big) \ar[dd]_-{\mu} \ar[rr]^-{\id \times \prod_r(\id \times \prod_s\tha_{p_{rs}})} 
& & \sO(n)\times \prod_r(\sO(m_r)\times \sA^{m_r}) \ar[dd]_-{ \mu} \ar[dr]^(.65)*+<1ex>{^{\id \times\prod_r \tha_{m_r}}}  & \\
  & &  \drtwocell<\omit>{<0> \, \varphi}  & \sO(n)\times \sA^n \ar[dd]^-{\tha_n}\\
\sO(m)\times \prod_{r,s}(\sO(p_{rs})\times \sA^{p_{rs}}) \ar[dr]_-{\mu} \ar[rr]^-{\id \times \prod_{r,s} \tha_{p_{rs}} }  
& \drtwocell<\omit>{<0> \, \varphi}  &  \sO(m)\times \sA^m \ar[dr]^-{\tha_{m}} & \\ 
 & \sO(p)\times \sA^p \ar[rr]_-{\tha_p} &    & \sA} 
\] 
\[\xymatrixcolsep{-.3cm}\xymatrix{
\sO(n)\times \prod_r\big(\sO(m_r)\times \prod_s(\sO(p_{rs}) \times \sA^{p_{rs}})\big)  \ar[dd]_-{\mu} \ar[rr]^-{\id \times \prod_r(\id \times \prod_s\tha_{p_{rs}})} \ar[dr]^-{\id \times \mu}   
& \drtwocell<\omit>{ \qquad  \id \times \varphi} & \sO(n)\times \prod_r(\sO(m_r)\times \sA^{m_r})  \ar[dr]^(.65)*+<1ex>{^{\id \times\prod_r \tha_{m_r}} }
& \\
 & \sO(n)\times \prod_r (\sO(p_r)\times \sA^{p_r})  \ar[dd]_-{\mu} \ar[rr]_-{\id \times \prod_r \tha} \ddrrtwocell<\omit>{<0> \, \varphi}  &  &  \ \sO(n)\times \sA^n\ \ \ \ \ \   \ar[dd]^-{\tha_n}\\
 \sO(m)\times \prod_{r,s}(\sO(p_{rs})\times \sA^{p_{rs}}) \ar[dr]_-{\mu} 
&  & &   \\
 & \sO(p)\times \sA^p \ar[rr]_-{\tha_p} & & \sA   .
}\]
}}
This axiom is the translation of the equality of pasting diagrams specified in \autoref{Tpseudo}.
\end{ax}

If the transformations $\varphi$ are all the identity, then all axioms are satisfied automatically, and $\sA$ is a (strict) $\sO$-algebra as originally defined in \cite[\S1]{MayGeo}. 
\end{defn}

It is clear from the definition that $\sA$ is an $\sO$-pseudoalgebra if and only if it is a normal 
$\bO$-pseudoalgebra. The two Operadic Identity properties are precisely what is needed to
give the normality.

\begin{defn}\label{Opseudo2} An {\em $\sO$-pseudomorphism}  $(f,\ze)\colon (\sA, \tha, \varphi)$ and $(\sB,\xi,\ps)$ of $\sO$-pseudo\-algebras is given by a $\sV$-functor $f\colon \sA\rtarr \sB$  and a sequence of
 invertible $\sV$-transformations $\ze_n$ 
\[\xymatrix{   \sO(n)\times \sA^n \ar[d]_{\tha_n } \ar[r]^{\id\times f^n}  \drtwocell<\omit>{<0> \, \ze_n} & \sO(n)\times \sB^n \ar[d]^{\tha_n } \\
     \sA \ar[r]_-{f} & \sB. \\} \]
We require $f$ to preserve $0$ and $1$, so that $\ze_0$ and the whiskering of $\ze_1$ with the map $1\times \id\colon \sA \cong\ast\times \sA \rtarr \sO(1)\times \sA$ are the identity.   Then $f$ is a based map, and hence it induces  a map $\bO f \colon \bO A \rtarr \bO B$. We moreover require 
\[\ze_n=\ze_n\com (\rho \times \rho^{-1})\]
for all $\rho\in \Sigma_n$. This implies that $\ze$ induces an invertible $\sV$-transformation
\[ \xymatrix{
\bO X \ar[r]^-{\bO f} \ar[d]_{\tha} \drtwocell<\omit>{<0> \, \ze} & \bO Y \ar[d]^{\xi}\\
X \ar[r]_-{f} & Y.\\} \]
We require the following two pasting diagrams to be equal.

{\footnotesize
\[ \xymatrix @C=1em{
\sO(n) \times \prod_r(\sO(m_r)\times \sA^{m_r}) \ar[rr]^-{\id \times \prod_r(\id \times f^{m_r})} \ar[dd]_{\mu} \ar[dr]^-{\id\times \prod \tha_{m_r}} &  \drtwocell<\omit>{<-.5>   \id\times\prod\ze_{m_r}\qquad \qquad \qquad} & \sO(n) \times \prod_r(\sO(m_r)\times \sB^{m_r}) \ar[dr]^{\id\times\prod\xi_{m_r}} & \\
 \drtwocell<\omit>{<0> \, \varphi} & \sO(n)\times \sA^n \ar[rr]^-{\id \times f^n} \ar[dd]_{\tha_m}  \ddrrtwocell<\omit>{<0> \, \ze_n} & & \ \sO(n)\times \sB^n\ \ \ \ \  \ \ \ar[dd]^{\xi_n} \\ 
\sO(m) \times \sA^m \ar[dr]_{\tha_m} & \\
 & \sA  \ar[rr]_{f} & & \sB} 
\]
\[ \xymatrix{
\sO(n) \times \prod_r(\sO(m_r)\times \sA^{m_r}) \ar[rr]^-{\id \times \prod_r(\id \times f^{m_r})}  \ar[dd]_{\mu} & & \sO(n) \times \prod_r(\sO(m_r)\times \sB^{m_r}) \ar[dr]^{\id\times\prod\xi_{m_r}} \ar[dd]_{\mu} \\
 &  & \ddrtwocell<\omit>{<-2> \, \ps} &\  \ \ \ \sO(n)\times \sB^n \ \ \ \ \ \  \ar[dd]^-{\xi_n} \\ 
 \sO(m)\times \sA^m \ar[dr]_-{\tha_m} \ar[rr]^-{\id \times f^m} &  \drtwocell<\omit>{\quad \ze_m} &  \sO(m)\times \sB^m \ar[dr]^-{\xi_m}   \\
  & \sA  \ar[rr]_-{f} & & \sB  \\} 
\]
}
The equality of these diagrams is equivalent to that of the pasting diagrams specified in \autoref{Tpseudo2}. If the $\ze_n$ are identity $\sV$-functors, then $f$ is a (strict) $\sO$-map.
\end{defn}

\begin{defn}\label{Opseudo3}  An {\em algebra 2-cell}
$\la \colon (f,\ze)\Longrightarrow (g,\ka)$ 
is given by a $\sV$-transforma\-tion $\la\colon f\Longrightarrow g$, not necessarily invertible, such that the pasting diagrams  specified in \autoref{Tpseudo2} are equal.  Explicitly, for all $n$
\[ \xymatrixcolsep{.75cm}\xymatrixrowsep{.75cm}\xymatrix{
\sO(n) \times \sA^n  \rrtwocell<5>^{\id \times f^n}_{\id \times g^n}{<0> \qquad \id \times \la^n}  \ar[dd]_-{\tha_n}  \ddrrtwocell<\omit>{<1.5> \quad \ka_n}&& \sO(n)\times \sB^n \ar[dd]^-{\xi_n}  && \sO(n) \times \sA^n \ar[dd]_{\tha_n}   \ar[rr]^{\id\times f^n}    \ddrrtwocell<\omit>{<-1.5> \quad \ze_n} && \sO(n) \times \sB^n \ar[dd]^{\xi_n} \\
&&&=\\
\sA \ar[rr]_g && \sB && \sA \rrtwocell<5>^{f}_{g}{<0> \, \la}  && \sB
}
\]
\end{defn}

As promised in the introduction, with these definitions, we have the three $2$-categories $\opa$, $\oap$, and $\oas$, and a comparison of definitions  identifies them with their monadic analogs  
$\bO\text{-}\bf{PsAlg}$, $\bO\text{-}\bf{AlgPs}$, and  $\bO\text{-}\bf{AlgSt}$.

\begin{rem} Since $\sV$ is cartesian monoidal, we have a diagonal map of operads $\DE\colon \sO\rtarr \sO\times \sO$.  
Use of $\DE$ shows that the $2$-category of $\sO$-pseudoalgebras is again cartesian monoidal, and it is also bicomplete.
\end{rem}

\begin{rem} We comment on paths not taken. As in \cite{DS}, we can define pseudo-operads 
by allowing the associativity diagram for the composition
functor $\ga$ to commute only up to $\sV$-isomorphism.  We can then define pseudoalgebras over pseudo-operads. 
Similarly, following \cite{Bat, DS}, we can define lax or op-lax $\sO$-algebras by not requiring the 
$\varphi$ to be isomorphisms.  For example, taking the operad to be the permutativity operad $\sP$ (see below), this defines lax symmetric monoidal categories.  Lax monoidal categories are studied in \cite{Bat, DS} and are called lax multitensors in \cite{BCW}.   
The papers \cite{Bat, DS} show that lax monoidal categories are strict algebras over an appropriate operad, and the same is also true of lax symmetric monoidal categories.  In the absence of applications, we prefer to ignore these further weakenings and this form of strictification.
\end{rem}

\section{Operadic specification of symmetric monoidal $G$-categories}\label{genuine} 

Except in \autoref{GV}, we specialize to the case  $\sV=\bf{Set}$ in this section.  However, we can use the functor $\bV\colon \bf{Set}\rtarr \sV$ from \autoref{bUbV} or its equivariant variant
 from \autoref{GbUbV} to generalize the basic definitions. Since $\bV$ preserves finite limits (see \autoref{finitelim}), it preserves groups and operads.
Applying $\bV$ to the operads defined below gives the corresponding operads in $\sV$ or $G\sV$, and their  algebras specify the analogues in $\sV$  or $G\sV$ of the algebraic structures we discuss.

We first recall the definition of the permutativity operad $\sP$, which is chaotic by definition.
We start with the associativity operad\footnote{Always denoted $\sM$ in previous work of the senior author.} $\mathbf{Assoc}$ in 
$\mathbf{Set}$, where $\mathbf{Assoc}(j) = \SI_j$ as a right $\SI_j$-set.  We write $e_j$ for the identity element of $\SI_j$.  We have block sum of permutations homomorphisms $\oplus\colon \SI_i\times \SI_j \rtarr \SI_{i+j}$. 
If $j= j_1+\cdots + j_k$ and $\si\in \SI_k$, we define $\si(j_1,\cdots, j_k)\in \SI_j$ to be the element 
that permutes the $k$ blocks of letters as $\si$ permutes $k$ letters. With these notations, the structure maps $\ga$ are given by\footnote{This corrects an
incorrect formula on \cite[p. 82]{MayPerm}.}
\[ \ga(\si; \ta_1,\cdots, \ta_j) = \si(j_1,\cdots,j_k)(\ta_1\oplus\cdots \oplus \ta_k). \]

This is forced by $\ga(e_k;e_{j_1},\cdots,e_{j_k}) = e_j$ and the equivariance formulas
\[ \ga(\si;\nu_1\ta_1,\cdots,\nu_k\ta_k) = \ga(\si;\nu_1,\cdots,\nu_j)(\ta_1\oplus\cdots \oplus \ta_k) \]
for $\nu_s\in \SI_{j_s}$ and
\[ \ga(\mu\si;\ta_1,\cdots,\ta_k) = \ga(\mu;\ta_{\si^{-1}(1)}, \cdots, \ta_{\si^{-1}(k)}) \si(j_1,\cdots,j_k) \]
for $\mu\in \SI_k$ 
 in the definition of an operad.  To see this, take $\mu = e_k$ and $\nu_s = e_{j_s}$ and use these formulas in order. Algebras over $\mathbf{Assoc}$ are monoids in $\bf{Set}$.  

\begin{defn}
Let $G$ be a discrete group.  Let $G$ act by right multiplication on $G$ and diagonally 
on $G\times G$.  With these actions on objects and morphisms, $\mathcal{E}G$ is a right $G$-category. It also has a left action via left multiplication, making it a $G$-category.
\end{defn}

As shown in \cite{GMM}, $B\mathcal{E}G$ is a universal principal $G$-bundle.
The permutativity operad $\sP$ is obtained by applying the product-preserving functor
$\mathcal{E}(-)$ to $\mathbf{Assoc}$.

\begin{defn} The {\em permutativity operad} $\sP$ is the chaotic categorification of $\mathbf{Assoc}$, so that
$\sP(j)$ is the right $\SI_j$-category $\mathcal{E}{\SI}_j$.  
\end{defn}

Clearly $\sP(0)$ and $\sP(1)$ are trivial, the latter with unique object $e_1= 1$.  
The structure map $\ga$ is induced from that of $\mathbf{Assoc}$ by application of  $\mathcal{E}(-)$.

There is a product-preserving functor $\mathbf{Cat}( \mathcal{E} G ,-)$ from the category of $G$-categories to itself.  It is considered in detail in \cite{ GM3, GMM}. 

\begin{defn}
Let $\sA$ be a $G$-category. Define $\mathbf{Cat}( \mathcal{E} G ,\sA)$  to be the $G$-category whose objects and morphisms 
are all (not necessarily equivariant) functors $ \mathcal{E} G  \rtarr \sA$ and all natural transformations between them.  The (left) action of $G$ on
$\mathbf{Cat}( \mathcal{E} G ,\sA)$ is given by conjugation.
\end{defn}

Note that, by \autoref{chaotic}, when $\sA$ is chaotic then so is $\mathbf{Cat}( \mathcal{E} G ,\sA)$.  Since the functor $\mathbf{Cat}( \mathcal{E} G ,-)$ preserves products, it also preserves structures defined in terms of products. In particular, it takes $G$-operads to $G$-operads.  The trivial $G$-functor 
$\mathcal{E} G \rtarr \ast$ 
is an equivalence of underlying categories and induces a $G$-functor 
\[ \io\colon \sA \rtarr \mathbf{Cat}( \mathcal{E} G ,\sA) \]
that is also an equivalence of underlying categories. It follows that,
on taking classifying spaces, $\io$ induces a nonequivariant homotopy equivalence.

\begin{defn} The {\em permutativity operad} $\sP_G$ in $\mathbf{Cat}(G{\text{-}}\mathbf{Set})$ is the chaotic operad 
$\sP_G = \mathbf{Cat}( \mathcal{E} G ,\sP)$, where $G$ acts trivially on $\sP$. Thus 
$\sP_G(j)$ is the $G$-category $\mathbf{Cat}( \mathcal{E} G ,\mathcal{E}{\SI}_j)$.
The operad structure is induced from that of $\sP$.  
\end{defn}

\begin{rem}\label{GV} Returning to a general $\sV$,
recall the category $G\sV$ of $G$-objects in $\sV$ from \autoref{DefnGV}
and the functor $\bV\colon G{\text{-}}\mathbf{Set}\rtarr G\sV$ from \autoref{GbUbV}.
Applying $\bV$, we regard $\sP_G$ as an operad in $\mathbf{Cat}(G\sV)$. In the case $\sV=\sU$, 
$\bV$ just gives a $G$-set the discrete topology. Thus, our notion of a symmetric monoidal $G$-category immediately extends to $G$-categories internal to $G$-spaces.
\end{rem}

Clearly $\sP_G$ is reduced and $\sP_G(1)$ is trivial with unique object $1$.  When $G$ is the 
trivial group, $\sP_G = \sP$. The functor $\io$ specifies a map $\sP\rtarr \sP_G$ of $G$-operads.
Application of $B$ gives a weak equivalence $B\sP\rtarr B\sP_G$ 
of nonequivariant operads.
The operad $B\sP_G$ is an equivariant $E_{\infty}$ operad, meaning that $B\sP_G(j)$ is a universal $(G,\SI_j)$-bundle (see \cite[Theorem 0.4]{GMM}).

It has been known since \cite{MayPerm} that $\sP$-algebras are the same as permutative categories, and in
\cite{ GM3} we defined a genuine permutative $G$-category to be a $\sP_G$-algebra.  In principle, for an 
operad $\sO$, $\sO$-algebras give ``unbiased'' algebraic structures.\footnote{Biased versus unbiased algebraic
structures are discussed in \cite[\S3.1]{Lein}, for example.}  Products $\sA^n\rtarr \sA$ are 
given for each object of $\sO(n)$.  Biased algebraic structures are defined more economically, usually
starting from a binary product $\mu\colon \sA\times \sA\rtarr \sA$.  Ignoring the associativity isomorphism
for cartesian products, the associativity axiom for permutative categories then states that 
$\mu\com (\mu\times \id)=\mu\com (\id\times \mu)$.   When permutative categories are defined by actions
of $\sP$, we are given a canonical $3$-fold product  $\sA^3 \rtarr \sA$, and the associativity axiom now says that
both $\mu\com (\mu\times \id)$ and $\mu\com (\id\times \mu)$ are equal to that $3$-fold product.  The 
biased definition of a permutative category requires use of only $\sA^n$ for $n\leq 3$.

Similarly, the biased definition of a symmetric monoidal category requires use of only $\sA^n$ for $n\leq 4$.
Use of four variables is necessary to state the pentagon axiom in the absence of strict associativity.
Just as permutative categories are the same as $\sP$-algebras, we claim that symmetric monoidal categories are essentially the same as $\sP$-pseudoalgebras.  

We have required the strict Operadic Identity Axiom on $\sP$-pseudoalgebras because that is both natural and necessary to our claim: symmetric monoidal categories $\sA$ come with the identity operation 
$\sA\rtarr \sA$, and there is nothing that might correspond to an isomorphism to the identity operation.   

More substantially, our Unit Object Axiom requires that $0$ be a strict unit object for the product on an 
$\sO$-pseudoalgebra.  This is of course not true for symmetric monoidal categories in general. The more precise claim is that symmetric monoidal categories with a strict unit object correspond bijectively to $\sP$-pseudoalgebras as we have defined them. This requires proof, which in one direction amounts to deriving the pentagon and hexagon axioms from the equivariance and associativity properties of the transformations $\varphi$ that appear in  the definition of $\sP$-pseudoalgebras, and in the other direction amounts to proving that, conversely, all the properties of the transformations $\varphi$ can be derived from those at lower levels.  Although not in the literature as far as we know, this is well-known categorical folklore and is left as an exercise.  See chapter 3 of \cite{Lein} for a discussion of the nonsymmetric case.

Of course, a symmetric monoidal category is monoidally equivalent to a symmetric monoidal category with a strict unit since it is monoidally equivalent to a permutative category, but the former equivalence is much easier to prove. It is a categorical analogue of growing a whisker to replace a based
space by an equivalent based space with nondegenerate basepoint \cite[\S5]{ GM2}.  Just as we require basepoints to be nondegenerate in topology, we require our symmetric monoidal categories to have strict unit objects. 

In Definition \ref{Gsymm}, we defined genuine symmetric monoidal $G$-categories to be $\sP_G$-pseudoalgebras, implicitly requiring them to satisfy our axioms.  The operadic definitions of genuine permutative and symmetric monoidal 
$G$-categories give unbiased algebraic structure, and here the biased notions have yet to be determined.

\begin{prob}  Determine biased specifications of genuine permutative and symmetric monoidal $G$-categories.
\end{prob}

That is, it is desirable to determine explicit additional structure on a naive permutative
or symmetric monoidal $G$-category that suffices to give it a genuine structure. As shown in \cite{Reedies} by the fourth author and her collaborators, this problem cannot be solved. More precisely, they show that if $G$ is a nontrivial finite group, the operad $\sP_G$ is not finitely generated. This means that in order to specify the structure of a $\sP_G$-algebra, one needs to specify an infinite amount of information, subject to an infinite number of axioms.

Using ideas from Rubin \cite{Rubin2}, one can produce a finitely generated suboperad $\sQ_G$ of $\sP_G$ that is equivariantly equivalent, in the sense that it is also an $E_\infty$ $G$-operad. Bangs et al. solve in \cite{Reedies} the problem of identifying biased specifications for algebras over $\sQ_G$ for $G=C_p$ when $p=2,3$. 

Rubin \cite{Rubin2} has solved this problem in a closely related but not identical context.  He proves a coherence
theorem of just the sort requested for algebras over the $N_{\infty}$ operads that he constructs.   Despite the close 
similarity of context, there is hardly any overlap between his work and ours.  His work in progress promises to establish
the precise relationship between our symmetric monoidal $G$-categories and commutative monoids in the relevant 
$G$-symmetric monoidal categories of Hill and Hopkins \cite{HH}.   Precisely, his normed symmetric monoidal categories 
are intermediate between these and will be compared to each in forthcoming papers of his.

Since naive permutative and symmetric monoidal $G$-categories are just nonequivariant structures with 
$G$ acting compatibly on all structure in sight, the nonequivariant equivalence between biased and unbiased 
definitions applies verbatim to them.  This has the following implication, which shows that naive
structures can be functorially extended to naively equivalent genuine structures.

\begin{prop} The functor $\mathbf{Cat}(\mathcal{E} G ,-)$ induces functors from naive to genuine 
permutative $G$-categories and from naive to genuine symmetric monoidal $G$-categories.  In both
cases, the constructed genuine structures are naively equivalent via $\io$ to the given naive
structures.
\end{prop} 

In particular, we can apply this to nonequivariant input categories or to categories with $G$-action.  Thus examples of genuine permutative 
and symmetric monoidal $G$-categories are ubiquitous.

\section{Enhanced factorization systems}\label{EFS0}
  
\subsection{Enhanced factorization systems}\label{EFS1}
In this section, we establish the context for the strictification theorem by defining enhanced factorization systems.  We let $\sK$ be an arbitrary $2$-category.

\begin{defn}\label{EFS} 
An {\em enhanced factorization system,} abbreviated EFS, on $\sK$ consists of a pair $(\mathcal{E},\sM)$ of classes of $1$-cells of $\sK$, both of which contain all isomorphisms, that satisfy the following properties.   
\begin{enumerate}[(i)]
\item Every $1$-cell $f$ factors as a composite 
\[ \xymatrix@1{ X \ar[r]^-{e_f} & \bI f \ar[r]^-{m_f} & Y, \\} \]
where $m_f\in \sM$ and $e_f\in \mathcal{E}$. 
\item For a diagram in $\sK$ of the form
\[ \xymatrix{
A \ar[r]^{e} \ar[d]_-{v}  \drtwocell<\omit>{<0> \, \varphi}  & X \ar[d]^{u} \\
B \ar[r]_-{m} & Y,\\} \]
where $e\in \mathcal{E}$, $m\in \sM$, and $\varphi$ is an invertible $2$-cell, there is a unique 
pair $(w,\widetilde{\varphi})$ consisting of a $1$-cell $w\colon X\rtarr B$ and an invertible $2$-cell 
$\widetilde{\varphi}\colon u\Longrightarrow m\com w$ such that $w\com e = v$ and $\widetilde{\varphi} \com e = \varphi$.
\[ \xymatrix{
A \ar[rr]^-{e} \ar[d]_-{v} & & X \ar[d]^{u} \ar[dll]_{w}  \dtwocell<\omit>{<3>  {\widetilde{\varphi}} }  \\
B \ar[rr]_-{m}  & &  Y \\} \]
By the uniqueness, if $\varphi$ is the identity, then $u=m\com w$ and $\widetilde{\varphi}$ is the identity.
\item  With the notations of (ii), suppose that $m\com v = u\com e$ and a second pair of $1$-cells $v'\colon A \rtarr B$ and $u'\colon X \rtarr Y$ is given
such that $m\circ v' = u'\circ e$ together with 2-cells $\si \colon v \Rightarrow v'$ and $\ta \colon u\Rightarrow u'$ such that 
$\ta \circ e= m\circ \sigma$.  Then there exists a unique 2-cell $\rho \colon  w \Rightarrow w'$ such that $\rho \circ e = \si$ and 
$m\circ \rho = \ta$. 
\end{enumerate}
We say that an EFS $(\mathcal{E},\sM)$ is {\em rigid} if the following further property holds.
\begin{enumerate}[(iv)]
\item For $m\colon Y \rtarr X$ in $\sM$ and any 1-cell $f\colon X \rtarr Y$, if $m\circ f$ is 
isomorphic to $\id_X$ then 
$f\circ m$ is isomorphic to $\id_Y$.
\end{enumerate}
\end{defn}

It is important to notice that the notion of a rigid EFS depends only on the underlying $2$-category $\sK$ 
and not on any $2$-monad defined on it. 

\begin{rem} The notation $(\mathcal{E},\sM)$ reminds us of epimorphisms and monomorphisms,
which give the usual factorization system in the category of sets. The notation
$\bI f$ stands for the image of $f$, reminding us of this elementary intuition. The
maps in $\sM$ can be categorical monomorphisms, meaning that 
$m\com f = m\com g$ implies $f=g$ when $m\in \sM$.  However, this {\em fails} for the $\sM$ 
of primary interest in this paper.  We shall interpolate as remarks a number of results that
hold when the maps in $\sM$ are monomorphisms but that fail otherwise. 
\end{rem}

The following observation about factorizations of composites of $1$-cells illustrates how 
EFSs mimic the behavior of the image factorization of functions. Its proof is immediate from
\autoref{EFS}(ii) with a reorientation of its second diagram.

\begin{lem}\label{EFSComp} 
Let $(\cE,\sM)$ be an EFS on $\sK$.
For $1$-cells $f\colon X\rtarr Y$ and $g\colon Y\rtarr Z$, there is a
unique ``composition" $1$-cell $c\colon \bI(g f)\rtarr \bI g$ making the following diagram commute.
\[\xymatrix{ X \ar[r]^-{f} \ar[d]_{e_{gf}} & Y \ar[r]^-{e_g} & \bI g \ar[d]^{m_g} \\
 \bI(gf) \ar[rru]^{c} \ar[rr]_-{m_{gf}} &  &  Z \\} \]
\end{lem} 

\begin{rem}\label{EFSCompM}
If there is a $1$-cell $s\colon Z\to X$ such that $g f s = \id$ and if $m_g$ and $m_{g f}$
are categorical monomorphisms, then $c$ is an isomorphism with inverse 
$c^{-1} = e_{g f} s m_g$.
Indeed, given $s$,
\[  m_{gf} c^{-1}  c = m_{gf} e_{gf} s m_g c = g f s m_{gf} = m_{gf}\]
and 
\[ m_g c  c^{-1} = m_{g f} e_{g f} s m_g = g f s m_g = m_g. \]
The monomorphism property implies that $c^{-1} c = \id$ and $c c^{-1} = \id$. 
\end{rem}

We have an analogous observation about the factorization of products.

\begin{defn}\label{prod}  Let $\sK$ have products and an EFS $(\mathcal{E},\sM)$.  We say that $(\mathcal{E},\sM)$ is {\em product-preserving} if  the product of $1$-cells in $\mathcal{E}$ is in $\mathcal{E}$ and the product of $1$-cells in $\sM$ is in $\sM$.  The name is justified by the observation that if $(\mathcal{E},\sM)$ is product-preserving, then for each pair of $1$-cells $f\colon X\rtarr Y$ and $f'\colon X'\rtarr Y'$, application of \autoref{EFS}(ii) gives morphisms
\[ \xymatrix{  X \times X'  \ar[r]^-{e_{f\times f'}} \ar[d]_{e_{f} \times e_{f'}} &  \bI (f\times  f' )\ar[d]^{m_{f \times f'}}\\
\bI f\times \bI f' \ar[ur]^{p} \ar[r]_-{m_{f} \times m_{f'}}  & Y\times Y' }  \ \ \text{and} \ \   \xymatrix{  X \times X'  \ar[r]^-{e_f\times e_{f'}} \ar[d]_{e_{f\times f'}} &  \bI f\times \bI f' \ar[d]^{m_f \times m_{f'}}\\
\bI(f\times f') \ar[ur]^{p^{-1}} \ar[r]_-{m_{f\times f'}}  & Y\times Y'.} \]
By an argument similar to that in \autoref{EFSCompM}, these are inverse to each other.
\end{defn}

\subsection{The enhanced factorization system on $\bf{Cat}(\sV)$}\label{EFS2}

We need an enhanced factorization system on $\bf{Cat}(\sV)$.  The idea comes from Power's paper \cite{Power}. We owe the adaptation to our context to Nick Gurski.\footnote{Private communication.} 

\begin{defn}\label{defn:BOFF} A $\sV$-functor $f\colon \sX\rtarr \sY$ is {\em bijective on objects} if the $\sV$-map
$f\colon \bf{Ob}\sX\rtarr \bf{Ob}\sY$ is an isomorphism.  It is {\em full and faithful}
if the following square in $\sV$ is a pullback.
\[ \xymatrix{
\bf{Mor}\sX \ar[r]^-f \ar[d]_{(T,S)} & \bf{Mor}\sY \ar[d]^{(T,S)}\\
\bf{Ob}\sX\times \bf{Ob}\sX \ar[r]_-{f\times f} & \bf{Ob}\sY\times \bf{Ob}\sY \\} \]
We abbreviate by calling the class of functors that are bijective on
objects $\sB\sO$ and calling the class of functors that are full and faithful $\sF\sF$.
\end{defn} 

\begin{lem}\label{prod2}  
The classes $\sB\sO$ and $\sF\sF$ of $\sV$-functors are closed under products.
\end{lem}
\begin{proof}  A product of isomorphisms is an isomorphism and products commute with pullbacks.
\end{proof}    

We defer the proof of the following purely categorical theorem to \S\ref{EFS3}.

\begin{thm}\label{EFSCatV} The classes $(\sB\sO,\sF\sF)$ specify a product-preserving rigid enhanced 
factorization system on $\bf{Cat}(\sV)$.
\end{thm}

The full and faithful functors are {\em not} categorical monomorphisms, but they do satisfy an illuminating weaker condition.

\begin{lem}\label{mono}  Assume that $f\colon \sX \rtarr \sY$ is a full and faithful $\sV$-functor and that
 $g,h\colon \sZ \rtarr \sX$ are 
$\sV$-functors such that $g = h\colon  \bf{Ob}\sZ \rtarr  \bf{Ob}\sX$  and $fg = fh$.  Then
$g=h \colon  \bf{Mor}\sZ \rtarr  \bf{Mor}\sX$ and thus $g=h$.
\end{lem}
\begin{proof}   This is a direct application of the universal property of pullbacks.
\end{proof}

We can apply this to obtain a weakened  modification of \autoref{EFSComp}.  This is a digression since even the modification fails to apply in our applications, but it may well apply in other situations.

\begin{lem}\label{EFSComp2}
Let $f\colon \sX\rtarr \sY$ and $g\colon \sY\rtarr \sZ$ be $\sV$-functors.  If there is a $\sV$-functor
$s\colon \sZ\rtarr \sX$ such that  $sgf = \id$ on $\mathbf{Ob}X$, $fsg = \id$ on $\mathbf{Ob}Y$, and
$gfs =\id$, then the composition $\sV$-functor 
$c\colon \bI(gf) \rtarr \bI g$ of \autoref{EFSComp} is an isomorphism with inverse 
$c^{-1} = e_{gf} s m_g$.
\end{lem}

\begin{proof}  In the proof of \autoref{EFSCatV}, we take $\mathbf{Ob}\bI f = \mathbf{Ob}\sX$ and take $e_f = \id$ and $m_f = f$ on objects.  Therefore $c=f$ on objects.  The hypotheses on objects ensure that \autoref{mono} applies to give the conclusion.
\end{proof}

\subsection{Proof of the properties of the EFS on $\bf{Cat}(\sV)$}\label{EFS3}

We break the proof of \autoref{EFSCatV} into a series of propositions.  

\begin{prop}\label{enfact1} Every $\sV$-functor $f\colon \sX\rtarr \sY$ factors 
as a composite of $\sV$-functors 
\[ \xymatrix@1{ \sX \ar[r]^-{e_f} & \bI f \ar[r]^-{m_f} & \sY,\\}  \]
where $e_f$ is in $\sB\sO$ and $m_f$ is in $\sF\sF$.  Moreover, the factorization 
commutes with products.
\end{prop}

\begin{proof}   Define $\bf{Ob}(\bI f) = \bf{Ob}\sX$ and define 
\[ e_f = \id\colon  {\bf{Ob}}\sX\rtarr {\bf{Ob}}(\bI f) \ \ \text{and} \ \ m_f = f \colon \bf{Ob}(\bI f)\rtarr \bf{Ob}\sY. \]
Define $\bf{Mor}(\bI f)$ via the pullback square displayed in the diagram
\[ \xymatrix{
\bf{Mor}\sX \ar[ddr]_{(T,S)} \ar[drr]^-{f} \ar[dr]_{e_f} & & \\
& \bf{Mor}(\bI f) \ar[r]_-{m_f} \ar[d]^{(T,S)} & \bf{Mor}\sY \ar[d]^{(T,S)}\\
& \bf{Ob}\sX\times \bf{Ob}\sX \ar[r]_-{f\times f} & \bf{Ob}\sY\times \bf{Ob}\sY \\} \]
The diagram displays both the definition of $m_f$ on $\bf{Mor}(\bI f)$ and the construction of $e_f$ on 
$\bf{Mor}\sX$ by application of the universal property of pullbacks. It also displays the definition of $T$ 
and $S$ on $\bf{Mor}(\bI f)$.  Clearly $e_f$ is bijective on objects and $m_f$ is full and faithful, as specified in \autoref{defn:BOFF}.
We must define 
\[ I = e_f\com I\colon \bf{Ob}\sX = \bf{Ob}(\bI f) \rtarr \bf{Mor}(\bI f) \]
for $e_f$ to commute with $I$, and then $m_f$ commutes with $I$ since $f$ commutes with $I$. Noting that the
outer square commutes, the following diagram displays the definition of $C$ in $\bI f$ by application of 
the universal property of pullbacks.
\[ \xymatrix{
\bf{Mor}(\bI f)\times_{\bf{Ob}(\bI f)}\bf{Mor}(\bI f)  \ar[dd]_{T\times S} \ar[rr]^-{m_f\times m_f} \ar[dr]^{C} 
& & \bf{Mor}\sY\times_{\bf{Ob}\sY}\bf{Mor}\sY \ar[d]^C  \\
& \bf{Mor}(\bI f) \ar[r]_-{m_f} \ar[dl]^{(T,S)} & \bf{Mor}\sY \ar[d]^{(T,S)}\\
\bf{Ob}\sX\times \bf{Ob}\sX \ar[rr]_-{f\times f} & & \bf{Ob}\sY\times \bf{Ob}\sY \\} \]
Using that 
\[ (f\times f)\com (T, S) \com C = (T, S)\com C\com (f\times f)\colon 
\bf{Mor}\sX\times_{\bf{Ob}\sX}\bf{Mor}\sX \rtarr \bf{Ob}\sY\times \bf{Ob}\sY, \]
it follows that 
\[ C\com (e_f\times e_f) = e_f\com C\colon \bf{Mor}\sX\times_{\bf{Ob}\sX}\bf{Mor}\sX\rtarr \bf{Mor}(\bI f).\]
That $C$ is associative and unital on ${\bI f}$ follows from the pullback definition of $C$. 
\end{proof}

\begin{prop}\label{enfact2}
For a diagram in $\bf{Cat}(\sV)$ of the form
\[ \xymatrix{
\sA \ar[r]^{e} \ar[d]_-{v} \drtwocell<\omit>{<0> \, \varphi} & \sX \ar[d]^{u} \\
\sB \ar[r]_-{m} & \sY,  \\} \]
where $e$ is in $\sB\sO$, $m$ is in $\sF\sF$, and $\varphi$ is an invertible $\sV$-transformation, 
there is a unique 
$\sV$-functor $w\colon \sX\rtarr \sB$ and a unique invertible $\sV$-transformation $\widetilde{\varphi}\colon u\Longrightarrow m\com w$ 
such that $w\com e = v$ and $\widetilde{\varphi} \com e = \varphi$. 
\[ 
\xymatrix{
\sA \ar[rr]^-{e} \ar[d]_-{v} & & \sX \ar[d]^{u} \ar[dll]_{w}  \dtwocell<\omit>{<3>  \widetilde{\varphi} }  \\
\sB \ar[rr]_-{m}  & &  \sY \\} 
\]
\end{prop}
\begin{proof}
Since $e \in \sB\sO$, we can and must define $w$ and $\widetilde{\varphi}$ on objects by
\[ w = v\com {e}^{-1}\colon \bf{Ob}\sX \rtarr \bf{Ob}\sB \]
and
\[ \widetilde{\varphi} = \varphi\com {e}^{-1}\colon \bf{Ob}{\sX} \rtarr \bf{Mor}\sY. \]
We define $w$ on morphisms by use of the following diagram, noting that the
lower right square is a pullback since $m$ is in $\sF\sF$. 

{\small

\[ \xymatrix{
\bf{Ob}\sX\times \bf{Mor}\sX \times \bf{Ob}\sX \ar[rr]^-{e^{-1}\times \id \times e^{-1}} & &
\bf{Ob}\sA\times \bf{Mor}\sX \times \bf{Ob}\sA \ar[d]^{\varphi\times u \times \varphi^{-1}}\\
\bf{Mor}\sX \ar[d]_{(T,S)} \ar[u]^{(T,\id,S)}  \ar[dr]^{w}
& & \bf{Mor}\sY\!\times_{\bf{Ob}\sY}\! \times \bf{Mor}\sY\!\times_{\bf{Ob}\sY}\! \bf{Mor}\sY \ar[d]^C \\
\bf{Ob}\sX\times \bf{Ob}\sX \ar[d]_{e^{-1}\times e^{-1}} \ar[dr]^{w\times w} & \bf{Mor}\sB \ar[d]_{(T,S)}\ar[r]^m
& \bf{Mor}\sY \ar[d]^{(T,S)}\\
\bf{Ob}\sA\times \bf{Ob}\sA \ar[r]_-{v\times v} 
& \bf{Ob}\sB\times \bf{Ob}\sB \ar[r]_-{m\times m} & \bf{Ob}\sY\times \bf{Ob}\sY
\\}\]

}

\noindent
Using that on objects, $u = u\com e\com e^{-1}$, $S\com \varphi = u\com e$, $T\com \varphi = m\com v$,
and thus $S\com \varphi^{-1} = m\com v$, and $T\com \varphi^{-1} = u\com e$, we see that $C$ is well-defined
and the outer rectangle commutes; for the latter we also use the axioms of a $\sV$-transformation. 
The universal property of pullbacks gives $w$. The diagram shows that $w$ commutes with $S$ and $T$, 
and it is easy to check that it also commutes with $I$ and $C$. Precomposing with
$e\colon \bf{Mor}\sA \rtarr \bf{Mor}\sX$ and using the coherence axioms for $\sV$-transformations and the universal 
property of pullbacks, we see that $w\com e = v$ on morphisms, so that there is an equality
$w\com e = v$ of $\sV$-functors.  Clearly 
\[ S\com \widetilde{\varphi} = S\com \varphi \com e^{-1} = u\com e \com e^{-1} = u \]
and 
\[ T\com \widetilde{\varphi} = T\com \varphi \com e^{-1} = m\com v\com e^{-1} = m\com w.\] 
We check that the diagram (\ref{vtrans2}) required of a $\sV$-transformation commutes, 
so that $\widetilde{\varphi}\colon u\Longrightarrow m\com w$, by use of the pullback definition of $w$ on morphisms. 
It is also easy to check that $\widetilde{\varphi}$ is invertible 
with inverse $\varphi^{-1}\com e^{-1}\colon \bf{Ob}\sX\rtarr \bf{Mor}\sY$.  The uniqueness of $w$ on 
morphisms follows from the pullback description of $\bf{Mor}\sB$, checking that $w$ must have the
maps to $\bf{Mor}\sY$ and $\bf{Ob}\sB\times \bf{Ob}\sB$ displayed in the diagram above.
\end{proof}

\begin{prop}\label{enfact3}
With the notation of \autoref{enfact2}, suppose that $m\com v = u\com e$ and a second pair of $\sV$-functors $v'\colon \sA \rtarr \sB$ and $u'\colon \sX \rtarr \sY$ is given
such that $m\circ v' = u'\circ e$, together with $\sV$-transformations $\si \colon v \Rightarrow v'$ and $\ta \colon u\Rightarrow u'$ such that 
$\ta \circ e= m\circ \sigma$.  Then there exists a unique $\sV$-transformation $\rho \colon  w \Rightarrow w'$ such that 
\[
\xymatrix{
\sA \ar[rr]^e \ar[d]_e & & \sX \ar@{=}[dll] \ar[d]^w \dtwocell<\omit>{<3> \, \rho} \\
\sX \ar[rr]_{w'} & & \sB
} \quad \raisebox{-4.3ex}{=} \quad
\xymatrix{
\sA \ar[rr]^e \ar[d]_e \drrtwocell^v_{v'}{\sigma} & & X  \ar[d]^w \\
\sX \ar[rr]_{w'} & & \sB
} 
\]
and
\[
\xymatrix @R=2.05ex {
\sX \ar[rr]^w \ar[dd]_{w'} & \dtwocell<\omit>{<3> \, \rho} & \sB \ar@{=}[ddll] \ar[dd]^m  \\
 & & \\
\sB \ar[rr]_{m} & & \sY
} \quad \raisebox{-4.3ex}{=} \quad
\xymatrix {
\sX \ar[rr]^w \ar[d]_{w'} \drrtwocell^u_{u'}{\tau} & & \sB  \ar[d]^m \\
\sB \ar[rr]_{m} & & \sY.
} 
\]
\end{prop}

\begin{proof}
Since $\rho \circ e=\si$ and $e$ is bijective on objects, we must define $\rho$ by
\[\rho\:=\si \circ e^{-1}\colon \bf{Ob}\sX \rtarr \bf{Mor}\sB.\]
Then 
\[m\circ \rho = m\circ \si \circ e^{-1}=\ta \circ e \circ e^{-1} = \ta.\]
It remains to show that $\rho$ is indeed a $\sV$-transformation from $w$ to $w'$. The map $\rho$ satisfies \autoref{vtrans1}, since
\[S\circ \rho = S\circ \si \circ e ^{-1} = v\circ e^{-1}=w \quad \text{and} \quad T\circ \rho = T\circ \si \circ e ^{-1} = v'\circ e^{-1}=w'.\]

To prove that $\rho$ satisfies \autoref{vtrans2}, one uses that $m$ is full and faithful, and hence that $\bf{Mor}\sB$ is a pullback, to prove that the two maps $\bf{Mor}\sX \rtarr \bf{Mor}\sB$ are equal.
In more detail, it suffices to show that the two composites 
\[\xymatrix{
\bf{Mor} \sX \ar@<-.4ex>[r] \ar@<.4ex>[r]   & \bf{Mor}\sB \ar[r]^-{S,T,m} &  \bf{Ob}\sB \times \bf{Ob}\sB \times \bf{Mor}\sY
}\]
agree, and for this it suffices to show that the composites agree after applying each of the three projections. For the factor $S$, this follows from the equality $w = v\circ e^{-1}$ on objects, and a similar argument applies for the factor $T$. For the factor $m$, this follows from the equality $m\circ\rho = \tau$ and the naturality of $\tau$.
\end{proof}

\begin{prop}\label{enfact4} Suppose given $\sV$-functors $f\colon \sX\rtarr \sY$, $m\colon \sY\rtarr \sX$, 
and an invertible $\sV$-transformation $\io \colon \id \Longrightarrow m\com f$, where $m$ is in $\sF\sF$.
Then there is an invertible $\sV$-transformation $\nu\colon f\com m \Longrightarrow \id$.
\end{prop}

\begin{proof} The required $\sV$-natural transformation $\nu$ is given by the universal property of pullbacks 
applied in the diagram 
\[ \xymatrix{
\bf{Ob}\sY \ar[rr]^-{m} \ar[dr]^{\nu} \ar[ddr]_{(\id,f\com m)} & & \bf{Ob}\sX \ar[d]^{\io^{-1}}  \\
& \bf{Mor}\sY \ar[r]^-m \ar[d]^{(T,S)} & \bf{Mor}\sX \ar[d]^{(T,S)}\\
& \bf{Ob}\sY \times \bf{Ob}\sY  \ar[r]_{m\times m} & \bf{Ob}\sX \times \bf{Ob}\sX. \\} \]
where the outer parts of the diagram are easily seen to commute. The map $\nu$ satisfies \autoref{vtrans1} by the commutativity of the left triangle in the above diagram.
For the naturality diagram \eqref{vtrans2} of $\nu$, we must verify that two morphisms \mbox{$\bf{Mor} \sY \rtarr \bf{Mor} \sY$} agree. 
As in the proof of \autoref{enfact3}, since $m$ is in $\sF\sF$, it suffices to check that the two morphisms agree after applying $S$, $T$, and $m$. The cases of $S$ and $T$ are simple, and the case of $m$ follows from functoriality of $m$ and naturality of $\iota^{-1}$.

We obtain $\nu^{-1}$ similarly, replacing $\io^{-1}$ by $\io$ and $(\id,f\com m)$ by $(f\com m,\id)$
in the above diagram, and it is straightforward to check that the $\sV$-transformations $\nu$
and $\nu^{-1}$ are  inverse to each other.
\end{proof} 

In language to be introduced shortly (\autoref{IntEquiv}), the conclusion of \autoref{enfact4} can be promoted to the statement that $(f, m, \io,\nu)$ prescribes an internal equivalence between $\sX$ and $\sY$.

\section{The Power Lack strictification theorem}\label{PL}

\subsection{The statement of the strictification theorem}\label{PL1}

We return to an arbitrary $2$-category $\sK$.  We need some $2$-categorical preliminaries to make sense of the statement of the Power-Lack strictification theorem.

\begin{defn}\label{IntEquiv} An {\em internal equivalence} between objects ($0$-cells) $X$ and $Y$ of $\sK$
is given by $1$-cells $f\colon X\rtarr Y$ and $g\colon Y\rtarr X$
and invertible $2$-cells $\et\colon \id \Longrightarrow g\com f$ and 
$\epz\colon f\com g\Longrightarrow \id$; it is an adjoint equivalence if $\et$ and $\epz$ are the
unit and counit of an adjunction (the evident triangle identities hold). Given an (internal) equivalence 
$(f,g,\et,\epz)$,
we can replace $\epz$ by the composite
\[ \xymatrix@1{
f\com g \ar@{=>}[rr]^-{\epz^{-1}\com f\com g} & &f\com g \com f\com g \ar@{=>}[rr]^-{f\com \et^{-1}\com g}
& &  f\com g  \ar@{=>}[r]^-{\epz} & \id 
\\}\]
and so obtain an adjoint equivalence.\footnote{The proof is an elaboration of the proof of Theorem 1 in \cite[IV.4]{CatWork}.}
\end{defn}

The following observation is a variant of a result of Kelly \cite[\S3]{Kelly}. We give it in full
generality for consistency with the literature, but when we use it our conventions require us to restrict to  normal $\bT$-pseudoalgebras,  for which $\upsilon$ is the identity.

\begin{lem}\label{LiftBit} Let $\bT$ be a $2$-monad on $\sK$ and let 
\[ (f,\ze)\colon (X,\tha,\varphi,\upsilon)\rtarr (Y,\xi,\ps,\nu) \]
be a $\bT$-pseudomorphism between $\bT$-pseudoalgebras. If $f\dashv g$ is an adjoint 
equivalence in $\sK$, then the adjunction lifts to an adjoint 
equivalence in the $2$-category $\bT$-$\bf{PsAlg}$ of $\bT$-pseudoalgebras, 
$\bT$-pseudomorphisms, and algebra 2-cells.
By symmetry, the analogous statement with the roles of $f$ and $g$ reversed is also true.
\end{lem}
\begin{proof} Since $\ze$ is an invertible $2$-cell, we can define 
$\ka\colon \tha\com \bT g\Longrightarrow g\com \xi$ 
to be the composite $2$-cell
\[ \xymatrix @C=6em {
\bT Y \ar[r]^-{\bT g} \ar[d]_{\id} \dtwocell<\omit>{<-3> \, \bT \epz} & \bT X \ar[d]^{\tha} \ar[dl]^-{\bT f}  \ddtwocell<\omit>{<9> \, \ze^{-1}}\\
\bT Y \ar[d]_-{\xi}  &   X \ar[d]^{\id} \ar[dl]^-f \dtwocell<\omit>{<4> \, \eta} \\
 Y \ar[r]_g & X \\} \]
Diagram chases show that $(g,\ka)$ is a morphism of $\bT$-pseudoalgebras 
and that $\et$ and $\epz$ are algebra $2$-cells. 
\end{proof}

\begin{defn}
Let $(\cE,\sM)$ be an EFS on $\sK$.
A monad $\bT$ on $\sK$ is said to {\em preserve} $\cE$ if whenever $e$ is a $1$-cell in $\cE$, then $\bT e$ is also a $1$-cell in $\cE$.
\end{defn} 

We repeat the statement of the strictification theorem for the reader's convenience. 

\begin{thm}\label{Lack} Let $\sK$ have a rigid enhanced factorization system $(\mathcal{E},\sM)$ and let $\bT$ be a monad in $\sK$ which preserves $\cE$.
Then the inclusion of $2$-categories
\[ \bJ\colon \ets \rtarr \tpa  \]
has a left 2-adjoint strictification $2$-functor $\mathbf{St}$, and the component of the unit 
of the adjunction is an internal equivalence in $\tpa$.
\end{thm}

\begin{rem}\label{counitissue}
For a strict $\bT$-algebra $X$, the counit $\epz\colon \mathbf{St} \bJ X\rtarr X$ is a map in $\ets$, but when we view it via $\bJ$ as a map in $\tpa$, it is an internal equivalence, with inverse given by the unit.  Note that the counit is not necessarily an equivalence in $\ets$. 
\end{rem}

To apply \autoref{Lack} to prove \autoref{ConjOut1}, it remains only to verify its hypothesis on
the relevant monads.

\begin{prop}  For any operad $\sO$ in $\bf{Cat}(\sV)$, the monad $\bO$ of \autoref{OMon}
preserves $\sB\sO$. 
\end{prop}
\begin{proof}  
The objects of the categories $\sO(n)$ give an operad $\bf{Ob}\sO$ in $\sV$.
Since $\bf{Ob}$ is a left adjoint (with right adjoint the chaotic
category functor of \autoref{chaotic0}) and a right adjoint (with left adjoint the discrete category functor, see \autoref{Vembed}) it commutes with colimits and limits.  It follows that the monad $\bf{Ob}\bO$ associated to the operad $\bf{Ob}\sO$
satisfies 
$(\bf{Ob}\bO)(\bf{Ob}\sX)\iso \bf{Ob}(\bO \sX)$. Since any functor, such as $\bf{Ob}\bO$, 
preserves isomorphisms, the conclusion follows for $\bO$. 
\end{proof}

\subsection{The construction of the $2$-functor $\bf{St}$}\label{PL2}

We give the definition of $\bf{St}$ on $0$-cells, $1$-cells, and $2$-cells here and fill in details of omitted proofs in the following section. Given a $\bT$-pseudoalgebra $(X,\tha,\varphi,\upsilon)$, we obtain $\bf{St} X$ as $\bI \tha$. Explicitly,
we factor $\tha$ as the composite
\[\xymatrix{
X\ar[r]^{e_\tha} & \mathbf{St} X \ar[r]^{m_\theta} & X,
}\]
where $e_{\tha}$  is in $\mathcal{E}$ and  $m_{\tha}$ in $\sM$. Noting our assumption that $\bT e_{\tha}$ is in $\mathcal{E}$ and applying \autoref{EFS}(ii) to $\varphi\colon \tha\com \bT \tha \Longrightarrow \tha\com \mu$, we obtain a diagram
\begin{equation}\label{Stone} 
\xymatrix{
\bT\bT X \ar[r]^-{\bT e_\tha} \ar[d]_-{\mu} & \bT \mathbf{St} X \ar[d]^{\bT m_{\tha}} \ar[ddl]_(.45){\mathbf{St} \tha}    \\
\bT X \ar[d]_{e_\tha} & \bT X \ar[d]^\tha  \dtwocell<\omit>{<4>  {\widetilde{\varphi}} } \\
\mathbf{St} X \ar[r]_{m_{\tha}}  &  X \\ }
\end{equation}
in which $\mathbf{St}\tha\com \bT e_{\tha} = e_{\tha}\com \mu$ and $\widetilde{\varphi}\com \bT e_{\tha} = \varphi$.  

\begin{lem}\label{strict}  Let $(X,\tha,\varphi,\upsilon)$ be a $\bT$-pseudoalgebra.  Then
$(\mathbf{St}X,\mathbf{St}\tha)$ is a strict $\bT$-algebra and 
$(m_{\tha},\widetilde{\varphi})\colon (\mathbf{St}X,\mathbf{St} \tha,\id,\id)\rtarr (X,\tha,\varphi,\upsilon)$ is a $\bT$-pseudomorphism. If $(X,\tha)$ is a strict $\bT$-algebra, then $m_{\tha}\colon (\mathbf{St} X,\mathbf{St} \tha)\rtarr (X,\tha)$ is a strict $\bT$-map.
\end{lem}

\begin{rem} 
 The construction of $\bf{St}$ specializes as follows. Given an $\sO$-pseudo\-algebra  $(\sX,\tha,\varphi)$ in 
 $\bf{Cat}(\sV_*)$, thought of as a normal $\bO$-pseudoalgebra, the strictification $\bf{St} \sX$ is the $\sV_*$-category in the factorization
 \[\xymatrix{
\bO \sX\ar[r]^{e_\tha} & \mathbf{St} \sX \ar[r]^{m_\theta} & \sX.
}\]
Using the explicit construction of the factorization in \autoref{enfact1}, we see that 
\[\bf{Ob}(\bf{St}\sX)=\bf{Ob}(\bO \sX)\] 
and $\bf{Mor}(\bf{St}\sX)$ is constructed as the pullback
\[ \xymatrix{
 \bf{Mor}(\bf{St} \sX) \ar[r]^-{m_\tha} \ar[d]_{(T,S)} & \bf{Mor}\sX \ar[d]^{(T,S)}\\
 \bf{Ob}(\bO\sX)\times \bf{Ob}(\bO\sX) \ar[r]_-{\tha\times \tha} & \bf{Ob}\sX\times \bf{Ob}\sX. \\} \]
For example, if $\sV=\bf{Set}$ and $\sO=\sP$, the based category $\bf{St}\sX$ has objects given by 
 $n$-tuples of objects in $\sX$, restricting to non-base objects if $n>1$. A morphism $(x_1,\dots,x_n) \rtarr (y_1,\dots,y_m)$ is given by a morphism $\tha(x_1,\dots,x_n) \rtarr \tha (y_1,\dots,y_m)$ in $\sX$. 
 
 If instead we consider the strictification when considering $\sX$ as an $\bO_+$-pseudo\-algebra, we obtain a category whose set of objects is the free associative monoid on $\bf{Ob}\, \sX$, i.e., the objects are $n$-tuples of objects in $\sX$, and morphisms are defined similarly. This latter case recovers the classical strictification due to Isbell \cite{Is}.  \end{rem}

\begin{rem}
For a strict $\bT$-algebra $X$, the strict $\bT$-map $m_{\tha}\colon \mathbf{St}\bJ X\rtarr X$ specifies the component at $X$ of the counit $\epz$ of the adjunction claimed in \autoref{Lack}.
\end{rem}

We next define $\bf{St}$ on $1$-cells.  Using generic notation for structure maps, let
$$(f,\ze)\colon (X,\tha,\varphi,\upsilon)\rtarr (Y,\tha,\varphi,\upsilon)$$ 
be a $\bT$-pseudomorphism.  Applying \autoref{EFS}(ii) to $\ze^{-1}\colon f\com \tha\Longrightarrow \tha \com Tf$, we obtain a diagram 
\begin{equation}\label{Sttwo} 
 \xymatrix{
\bT X \ar[d]_-{\bT f} \ar[r]^{e_{\tha}} & \mathbf{St} X \ar[r]^{m_{\tha}} \ar[d]_(.45){\mathbf{St}{f}} \drtwocell<\omit>{  \xi } & X \ar[d]^{f}   \\
\bT Z \ar[r]_-{e_\tha} &  \mathbf{St}Y \ar[r]_{m_\tha} & Y
 } 
\end{equation}
in which $\mathbf{St} f\com e_{\tha} = e_{\tha} \com Tf$ and $\xi\com e_{\tha} = \ze^{-1}$. 

\begin{lem}\label{strict2} $\bf{St} f$ is a strict $\bT$-morphism for any $\bT$-pseudomorphism $(f,\ze)$.
\end{lem}

For a $\bT$-pseudoalgegbra $X$, define $k \colon X\rtarr \mathbf{St} X$ to be the composite
\[ \xymatrix@1{X \ar[r]^-{\io_X} &  \bT X \ar[r]^-{e_{\tha}} & \mathbf{St} X.\\} \] 
Since $m_{\tha}\com k = \tha\com \io_X$, we have the 
invertible unit $2$-cell $\upsilon\colon \id_X \Longrightarrow m_{\tha}\com k$. By the rigidity assumption of
\autoref{EFS}(iv), there is an invertible 
$2$-cell $\nu\colon k\com m_{\tha} \Longrightarrow \id$. As observed in \autoref{IntEquiv}, 
we may choose $\nu$ so that $(m_{\tha},k, \upsilon,\nu)$ is an adjoint equivalence in $\sK$.

Since $(m_{\tha},\widetilde{\varphi})$ is a $\bT$-pseudomorphism, the last statement of \autoref{LiftBit} shows that we can construct an invertible $2$-cell $\om \colon \mathbf{St}\tha \com \bT k \Longrightarrow k\com \tha$ such that $(k,\om)$ is a $\bT$-pseudomorphism $X\rtarr \bJ\bf{St} X$ and
the adjunction $k\dashv m_{\tha}$ lifts to an adjoint equivalence of $\bT$-pseudoalgebras.  

\begin{rem} For a $\bT$-pseudoalgebra $X$, the $\bT$-pseudomorphism $(k,\om)$ is the component of the unit $\et_X\colon X \rtarr \bJ\bf{St}X$ of the $2$-adjunction claimed in \autoref{Lack}, and we have just verified that it is an adjoint equivalence. 
\end{rem}

\begin{rem}  Expanding on \autoref{counitissue}, for a strict $\bT$-algebra $X$, the inverse in $\tpa$ of the strict $\bT$-map $\epz_X$, thought of as the $\bT$-pseudomap $\bJ\epz_X$, is $\et_{\bJ X}\colon \bJ X \rtarr \bJ\bf{St}\bJ X$. Even when $X$ is given as a strict algebra, $\om$ is not necessarily the identity. That is why the counit is only an internal equivalence in $\tpa$, not in $\ets$.
\end{rem}

The following remark about when $\epz$ is an internal equivalence in $\ets$ plays a key role in the categorical literature in general and in some of our applications.  

\begin{rem}\label{flexible}   
With the terminology of Blackwell, Kelly, and Power \cite[\S4]{BKP}, 
a strict $\bT$-algebra $X$ is said to be {\em semi-flexible} if $\epz$ is an equivalence in 
$\ets$ and to be {\em flexible} if $\epz$ is a retraction in $\ets$.   If $X = \mathbf{St} Z$ for a $\bT$-pseudoalgebra $Z$, then $X$ is flexible, as observed in \cite[Remark 4.5]{BKP}.  
Indeed, if $\et_Z\colon Z\rtarr \bJ\mathbf{St} Z$ is the unit, then $\mathbf{St} \eta_Z\colon \mathbf{St} Z \rtarr \mathbf{St} \bJ\mathbf{St} Z$ is an explicit strict map right inverse to $\epz_X$. In general, not all strict $\bT$-algebras are flexible or even semi-flexible, and not all flexible $\bT$-algebras are of the form 
$\mathbf{St} Z$.  Necessary and sufficient conditions for $X$ to be flexible or semi-flexible are given in 
\cite[Theorems 4.4 and 4.7]{BKP}.
\end{rem}

Finally, we define $\bf{St}$ on $2$-cells. Write  $k_X$ for the component of $k$ on $X$. For a $2$-cell 
$\si \colon (f,\ze) \Rightarrow (f',\ze')$ in $\tpa$, define the $2$-cell  
$\mathbf{St}\si$ to be the composite
\begin{equation}\label{Stthree}
\xymatrix{
& \mathbf{St} X \ar@/^1em/[rrd]^-{\mathbf{St}f}  & \\ 
\mathbf{St}X \ar[r]^{m_\tha}  \ar@{=}[ur] \ar@{=}[dr] \urtwocell<\omit>{<2> \hspace{1em} \nu^{-1}}  \drtwocell<\omit>{<-2>  \nu} 
& X \ar[u]_{k_X} \ar[d]^{k_X} \rtwocell^f_{f'}{<0> \, \si } & Y \ar[r]^{k_Y} & \mathbf{St} Y \\ 
& \mathbf{St} X \ar@/_1em/[rru]_-{\mathbf{St}f'}  & }
\end{equation}

We also comment on the interaction of $\mathbf{St}$ with products.
The product of $\bT$-pseudo\-algebras $(X,\tha)$ and $(Y,\tha')$ is a $\bT$-pseudoalgebra with action $\tha''$ given
by the composite
\begin{equation}\label{prodact}
 \xymatrix@1{   \bT(X\times X') \ar[r]^-{\pi} &  \bT X \times \bT X' \ar[r]^-{\tha\times \tha'}  & X\times X',\\} 
\end{equation}
where the components of $\pi$ are obtained by applying $\bT$ to the evident projections. 
Application of \autoref{EFSComp} to the composite (\ref{prodact}) gives the following addendum to \autoref{Lack}.  

\begin{cor}\label{Angelica}  If products of $1$-cells in $\sM$ are in $\sM$,
then there is a natural $1$-cell $\ga$ making the following diagram commute.
\[\xymatrix{  \bT(X\times X')  \ar[r]^-{\pi} \ar[d]_{e_{\tha''}} & \bT X \times \bT X'  \ar[r]^-{e_{\tha}\times e_{\tha'}} 
& \mathbf{St}X  \times \mathbf{St}X' \ar[d]^{m_{\tha}\times m_{\tha'}} \\
\mathbf{St}(X \times X')  \ar[rru]^{\ga} \ar[rr]_-{m_{\tha''}} &  & X\times Y \\} \]
\end{cor} 

\begin{rem} We shall not elaborate the details,
but the $2$-category of $\bT$-pseudo\-algebras is symmetric monoidal under $\times$, with unit the trivial object $\ast$, and
\autoref{Angelica} implies that $\mathbf{St}$ is an op-lax symmetric monoidal functor to the $2$-category of strict $\bT$-algebras.
\end{rem}

\begin{rem}  If, further, the $1$-cells in $\sM$ are monomorphisms, then the map $\ga$ of \autoref{Angelica} is an isomorphism.
Indeed, the map $\io\colon X\times Y \rtarr \bT(X\times X')$ satisfies $\tha''\com \io = \id$, hence \autoref{EFSCompM} applies.
\end{rem}

\subsection{The proof of the strictification theorem}\label{PL3}
We first prove the lemmas stated in the previous section and then give a shortcut
to the rest of the proof of  \autoref{Lack}.

\begin{proof}[Proof of \autoref{strict}] Since $\mu\com \io = \id$,   the uniqueness in \autoref{EFS}(ii)  implies that the $2$-cell composition
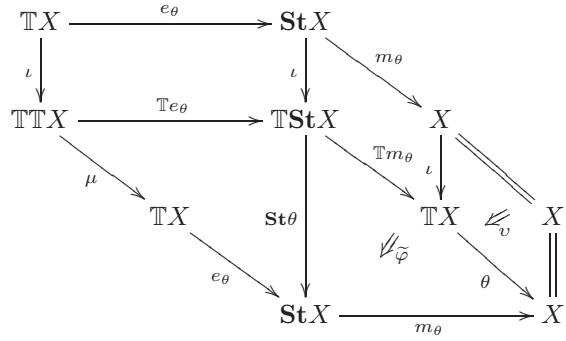
\[ \xymatrix{
\bT X \ar[rr]^-{e_\tha} \ar[d]_{\io} & &  \mathbf{St} X \ar[d]_{\io} \ar[dr]^{m_{\tha}} \\
\bT\bT X \ar[dr]_-{\mu} \ar[rr]^{\bT e_{\tha}} & & \bT \mathbf{St} X \ar[dr]^{\bT m_{\tha}} \ar[dd]_{\mathbf{St} \tha}   & X \ar@{=}[dr] \ar[d]_{\io} \ddrtwocell<\omit>{  \upsilon} \\
&  \bT X \ar[dr]_-{e_{\tha}} & \drtwocell<\omit>{<-2>  \widetilde{\varphi}}  & \bT X \ar[dr]_{\tha}  &  X \ar@{=}[d] \\
& & \mathbf{St} X \ar[rr]_{m_{\tha}} & &  X \\
}\]
is equal to the identity $2$-cell of $\bT X \xrightarrow{e_\tha} \mathbf{St} X \xrightarrow{m_\tha} X$.
Thus $\mathbf{St} \tha\com \io$ is the identity $1$-cell and the composite 
\[ \xymatrix@1{m_{\tha} \ar@{=>}[r]^-{\upsilon\com m_{\tha}} & 
\tha\com \io\com m_{\tha} = \tha\com \bT m_{\tha} \com \io \ar@{=>}[r]^-{\widetilde{\varphi}\com \io} & m_{\tha}\com \mathbf{St}\tha\com \io = m_{\tha}\\} \]
is the identity $2$-cell.

Similarly, the equality of pasting diagrams in \autoref{Tpseudo} and the uniqueness  in \autoref{EFS}(ii) imply that the $2$-cell composition

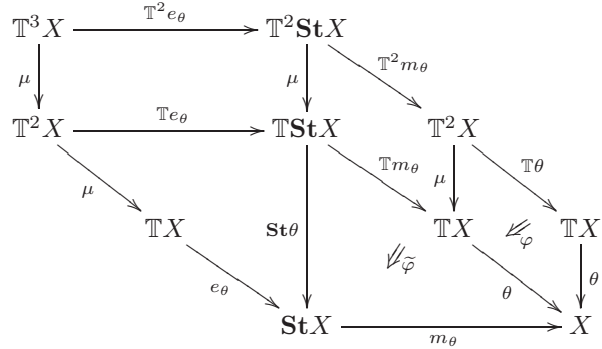
\[ \xymatrix{
\bT^3 X \ar[rr]^-{\bT^2 e_\tha} \ar[d]_{\mu} & &  \bT^2 \mathbf{St} X \ar[d]_{\mu} \ar[dr]^{\bT^2 m_{\tha}} \\
\bT^2 X \ar[dr]_-{\mu} \ar[rr]^{\bT e_{\tha}} & & \bT \mathbf{St} X \ar[dr]^{\bT m_{\tha}} \ar[dd]_{\mathbf{St} \tha}   & \bT^2 X \ar[dr]^{\bT \tha} \ar[d]_{\mu} \ddrtwocell<\omit>{  \varphi} \\
&  \bT X \ar[dr]_-{e_{\tha}} & \drtwocell<\omit>{<-2>  \widetilde{\varphi}}  & \bT X \ar[dr]_{\tha}  &  \bT X \ar[d]^\tha \\
& & \mathbf{St} X \ar[rr]_{m_{\tha}} & &  X \\
}\]
is equal to the $2$-cell composition
\[ \xymatrix{
\bT^3 X \ar[rr]^-{\bT^2 e_\tha} \ar[d]_{\bT\mu} & &  \bT^2 \mathbf{St} X \ar[d]_{\bT \mathbf{St} \tha} \ar[dr]^{\bT^2 m_{\tha}} \ddrtwocell<\omit>{ \quad \bT\widetilde{\varphi}}  \\
\bT^2 X \ar[dr]_-{\mu} \ar[rr]^{\bT e_{\tha}} & & \bT \mathbf{St} X \ar[dr]_{\bT m_{\tha}} \ar[dd]_{\mathbf{St} \tha}   & \bT^2 X \ar[dr]^{\bT \tha}   \\
&  \bT X \ar[dr]_-{e_{\tha}} & \drtwocell<\omit>{<-2>  \widetilde{\varphi}}  & \bT X \ar[dr]_{\tha} \ar@{=}[r]  &  \bT X \ar[d]^\tha \\
& & \mathbf{St} X \ar[rr]_{m_{\tha}} & &  X \\
}\]
Thus $\mathbf{St}\tha\com \mu = \mathbf{St}\tha\com \bT \mathbf{St} \tha$, so that $\mathbf{St} X$ is a strict $\bT$-algebra, and the implied equalities involving $\widetilde{\varphi}$ 
ensure that $(m_{\tha},\widetilde{\varphi})$ is a $\bT$-pseudomorphism. If $(X,\tha)$ is a strict $\bT$-algebra, then $\varphi$ and $\upsilon$ are identities, hence 
$m_{\tha}\com \mathbf{St} \tha = \tha \com \bT m_{\tha}$ and $\widetilde{\varphi}$ is the identity, showing that
$m_{\tha}\colon (\mathbf{St} X,\mathbf{St}\tha) \rtarr (X,\tha)$ 
is a strict $\bT$-map. 
\end{proof}

\begin{proof}[Proof of \autoref{strict2}] The equality of pasting diagrams 
given in the definition of a $\bT$-pseudomorphism \autoref{Tpseudo2} together with already indicated properties of our construction of $\textbf{St}$  imply that the following compositions of $2$-cells are equal. 
\[ \xymatrix{
&\bT^2 X \ar[r]^-{\bT e_{\tha}} \ar[dr]_-{\mu}  \ar[dl]_-{\bT^2f} &   
 \bT\mathbf{St} X \ar[r]^-{\bT m_{\tha}} \ar[rd]_-{\mathbf{St}\tha}  \drrtwocell<\omit>{\widetilde{\varphi}}& \bT X \ar[dr]^{\tha} &\\
\bT^2 Y\ar[dr]_{\mu} & &  \bT X \ar[r]^-{e_{\tha}} \ar[dl]_-{\bT f} & \mathbf{St} X \dtwocell<\omit>{\xi} \ar[ld]_-{\mathbf{St} f} \ar[r]^-{m_{\tha}} 
 & X \ar[dl]^f\\
& \bT Y \ar[r]_-{e_{\tha}}  & \mathbf{St} Y \ar[r]_-{m_{\tha}}  & Y & } \]
\[ \xymatrix{
& \bT^2 X \ar[dl]_-{\bT^2 f} \ar[r]^-{\bT e_{\tha}} & \bT \mathbf{St} X \ar[dl]_-{\bT\mathbf{St} f} \ar[r]^-{\bT m_{\tha}}  \dtwocell<\omit>{ \, \, \bT\xi}
&    \bT X \ar[rd]^-{\tha} \ar[dl]^-{\bT f} \ddtwocell<\omit>{ \quad \ze^{-1} } & \\
\bT^2 Y \ar[r]^-{\bT e_{\tha}} \ar[dr]_-{\mu} 
&  \bT \mathbf{St} Y \ar[dr]_-{\mathbf{St} \tha} \ar[r]^-{\bT m_{\tha}} \drrtwocell<\omit>{\widetilde{\varphi}}
& \bT Y \ar[dr]_-{\tha}  
&   & X \ar[dl]^-{f}   \\
& \bT Y \ar[r]_-{e_{\tha}} & \mathbf{St} Y \ar[r]_-{m_{\tha}}  & Y & } \]
This implies that $\mathbf{St} f\com \mathbf{St} \tha = \mathbf{St}\tha \com \bT\mathbf{St} f$, so that $\mathbf{St} f$ is a strict $\bT$-map.
\end{proof}

From here, diagram chases can be used to complete the proof of \autoref{Lack}.  For example, these show that $\mathbf{St}$ respects composition and identities at the level of 1-cells, so that we have a functor of the underlying categories, and that the triangle identities for the $2$-adjunction hold. The following categorical observation can be used to cut down substantially on the number of verifications required. It is a variant of \cite[Proposition 4.3.4]{Emily} in the enriched setting.

\begin{lem}\label{cute}
Let $\bJ\colon \sC \rightarrow \sD$ be 2-functor between $2$-categories. Suppose there exists a function on objects 
$F\colon \bf{Ob}\sD \rtarr \bf{Ob}\sC$ and for each object $d\in \sD$ a $1$-cell $\eta _d \colon d \rtarr \bJ Fd$ in $\sD$ 
such that for each object $c \in \sC$, applying $\bJ$ followed by precomposition with $\eta_d$ induces an isomorphism of categories
\[\nu\colon \sC(Fd,c) \rtarr \sD(d,\bJ c).\]
Then $F$ extends to a $2$-functor $F\colon \sD \rtarr \sC$ such that $F$ is left 2-adjoint to $\bJ$, with the 
unit of the adjunction given by $\eta$.
\end{lem}

\begin{proof} 
We define $F$ on $1$- and $2$-cells by the composite 
\[ \sD(d,d') \xrightarrow{\eta_{d'}\circ-} \sD(d,\bJ Fd') \xrightarrow{\nu^{-1}} \sC(Fd,Fd').\]
That $F$ is a $2$-functor such that $\eta$ is a $2$-natural transformation from the identity to $F$ follows formally from the 
definition.  For an object $c$ of $\sC$, the component at $c$ of the counit $\epz$ of the adjunction is the unique $1$-cell
$\epz_c\colon F\bJ c\rtarr c$ such that $\bJ \epz_c \com \eta_{\bJ c}$ is the identity of $\bJ c$. One triangle identity is obvious from the definition. The $2$-naturality of $\epz$ and the other triangle identity
follow from the uniqueness. 
\end{proof}

We apply this result to the inclusion $\bJ\colon \ets \rtarr \tpa$ and the construction of $\bf{St}$ on objects
given by \autoref{Stone} and \autoref{strict}. We must check that its hypothesis holds.  For a 
$\bT$-pseudoalgebra $X$, we take $\et_X = (k,\om)\colon X \rtarr \bJ\bf{St}X$.  For a $\bT$-pseudomorphism $(f,\ze)\colon (X,\tha,\varphi,\upsilon)\rtarr (Z,\tha)$, where $(Z,\tha)$ is a strict $\bT$-algebra, we define
$\widetilde f\colon \mathbf{St} X \rtarr Z$ to be the composite strict $\bT$-map
\[ \xymatrix@1{ \mathbf{St} X \ar[r]^-{\mathbf{St} f} & \mathbf{St} Z \ar[r]^-{m_{\tha}} & Z. \\} \]
It is straightforward to check that this map is the same as the one obtained by applying \autoref{EFS}(ii) to $\ze^{-1} \colon  f\com \tha \Longrightarrow \tha \com \bT f$: 
\[ \xymatrix{
\bT X \ar[d]_-{\bT f} \ar[r]^{e_{\tha}} & \mathbf{St} X \ar[r]^{m_{\tha}} \ar[d]_{\widetilde{f}} \drtwocell<\omit>{   } & X \ar[d]^{f}   \\
\bT Z \ar[r]_-{\tha} &  Z \ar@{=}[r] & Z \\
 } \]
Using this description, and using arguments similar to those in our proofs above, we can prove that $\widetilde{f}$ is the unique strict map such that $(\widetilde{f},\id)\com (k,\om) = (f,\ze)$. This gives the bijection of $1$-cells required for the isomorphism of categories 
\[ \ets(\mathbf{St} X,Z)\iso \tpa(X,Z)  \]
assumed in \autoref{cute}. The bijection at the level of $2$-cells follows from the fact that $(k,m_{\tha})$ is an internal adjoint equivalence. We can thus apply \autoref{cute} to finish the proof
of \autoref{Lack}.  \autoref{cute}  avoids the need to define $\mathbf{St}$ explicitly on 2-cells, to check that $\mathbf{St}$ is indeed a $2$-functor, and to check the $2$-naturality of $m_{\tha}$ and $(k,\om)$.  That is all given automatically.

Finally, we observed in \autoref{prod2} that the classes $\sB\sO$ and $\sF\sF$ are closed under products, so that $(\sB\sO,\sF\sF)$ is product-preserving.

\section{Appendix: strongly concrete categories}\label{AppRubin}

Recall the functor $\bV\colon \bf{Set} \rtarr \sV$ from \autoref{bUbV}. We prove here that it 
preserves finite limits under mild hypotheses that are satisfied in our examples. We must assume that $\sV$ has coproducts in addition to finite limits, and we assume further that the functors $V\times -$ and $-\times V$ preserve coproducts.  This is automatic if $\sV$ is cartesian closed, since these functors are then left adjoints.

\begin{lem}  The functor $\bV$ preserves finite products. 
\end{lem}
\begin{proof}  By definition, $\bV$ preserves $0$-fold products (terminal objects $\ast$),
and any functor preserves $1$-fold products, so it suffices to check that $\bV$ preserves
binary products. By our added hypothesis
\begin{eqnarray*} 
\bV S\times \bV T & = & \bigg(\coprod_{s\in S}\ast\bigg) \times \bV T  \iso  \coprod_{s\in S}(\ast\times \bV T)\\
&  = & \coprod_{s\in S} \bigg(\ast\times \coprod_{t\in T}\ast\bigg) \iso \coprod_{(s,t)} \ast\times\ast \iso  \bV(S\times T). \ \ \  \qedhere
\end{eqnarray*}
\end{proof}

Therefore $\bV$ preserves finite limits if it preserves equalizers. The following helpful definition
and proposition are due to Jonathan Rubin.\footnote{Private communication.}
Note that $\bV\emptyset$ is an empty coproduct and thus an initial object $\emptyset\in \sV$.

\begin{defn}\label{Conc} The category $\sV$ is {\em strongly concrete} if there is an 
underlying set functor $\bS\colon \sV\rtarr \mathbf{Set}$
with the following properties.
\begin{enumerate}[(i)]
\item There is a natural isomorphism $\Id\iso \bS\com \bV$.
\item The functor $\bS$ is faithful.
\item $\bS X = \emptyset$ if and only if $X = \emptyset$.
\end{enumerate}
Property (ii) says that $\sV$ is concrete in the usual sense.
\end{defn}

In many examples, we can take $\bS$ to be the right adjoint $\bU$ of $\bV$, and then (i) holds when the unit of the adjunction is an isomorphism (see \autoref{set}). However, this does not work in the equivariant context of most interest to us. 

\begin{exmp} Let $\sV = G\sU$.  For a set $S$, $\bV S$ is the the discrete space $S$ with trivial 
$G$-action.  The right adjoint $\bU$ of $\bV$ takes a $G$-space $X$ to the underlying set of $X^G$, and hence, $\bU$ does not satisfy (ii) and (iii) of \autoref{Conc}.
However, ignoring equivariance and taking $\bS X$ to instead be the 
underlying set of $X$, we see that $\bS$ satisfies all three conditions, so that $G\sU$ is strongly concrete.
\end{exmp}

\begin{prop} If $\sV$ is strongly concrete, then $\bV$ preserves equalizers and therefore all finite limits.
\end{prop}
\begin{proof}  Let 
$$ \xymatrix{ E\ar[r]^-{i} & S \ar@<.5ex>[r]^-{f} \ar@<-.5ex>[r]_-{g} & T\\} $$
be an equalizer in $\mathbf{Set}$.  We claim that 
$$ \xymatrix{ \bV E\ar[r]^{\bV i} & \bV S \ar@<.5ex>[r]^-{\bV f} \ar@<-.5ex>[r]_-{\bV g} & \bV T\\} $$
is an equalizer in $\sV$. By \autoref{Conc}(i), the given equalizer is isomorphic to 
$$ \xymatrix{ \bS \bV E\ar[r]^-{\bS\bV i} 
& \bS \bV S \ar@<.5ex>[r]^-{\bS\bV f} \ar@<-.5ex>[r]_-{\bS\bV g} & \bS\bV T\\}, $$
which is thus also an equalizer in $\mathbf{Set}$. 

Let $e\colon X\rtarr \bV S$ be a map in $\sV$ such that $\bV f\com e = \bV g\com e$.   We must show that there 
is a unique map $\widetilde{e}\colon X\rtarr \bV E$ such that $\bV i\com \widetilde e = e$. 
Since 
\[ \bS\bV f \com \bS e = \bS \bV g\com \bS e, \]
there is a unique map of sets 
$d\colon \bS X \rtarr \bS\bV E$
such that $\bS\bV i\com d = \bS e$. We claim that $d = \bS \widetilde e$ for a map $\widetilde e\colon X\rtarr \bV E$. 
Since $\bS$ is faithful and 
\[ \bS(\bV i\com \widetilde e) = \bS\bV i\com d = \bS e,\]
the claim implies both that $\bV i\com \widetilde e = e$ and that $\widetilde e$ is unique, completing the proof. 

Suppose first that $E\neq \emptyset$. Then, since $i$ is an injection, we can choose a map $r\colon S\rtarr E$ such that 
$r\com i = \id$. By inspection of set level equalizers, $d = \bS\bV r\com \bS e$, hence
$d =\bS\widetilde e$ where $\widetilde e = \bV r\com e$.

Finally, suppose $E=\emptyset$. Then $\bV E= \emptyset$ and, by \autoref{Conc}(iii), 
$\bS\bV E=\emptyset$. Thus $d$ is a map to $\emptyset$ and $\bS X = \emptyset$. By
\autoref{Conc}(iii) again, $X= \emptyset$ and we can and must let $\widetilde e$ be the unique map $\emptyset\rtarr \emptyset$.  
\end{proof}

\bibliographystyle{plain}
\bibliography{references}

\begin{thebibliography}{10}

\bibitem{And}
D.~W. Anderson.
\newblock Chain functors and homology theories.
\newblock In {\em Symposium on {A}lgebraic {T}opology ({B}attelle {S}eattle
  {R}es. {C}enter, {S}eattle, {W}ash., 1971)}, pages 1--12. Lecture Notes in
  Math., Vol. 249. Springer, Berlin, 1971.

\bibitem{Reedies}
K.~Bangs, S.~Binegar, Y.~Kim, K.~Ormsby, A.~M. Osorno, D.~Tamas-Parris, and
  L.~Xu.
\newblock Biased permutative equivariant categories.
\newblock Available as arXiv:1907.00933.

\bibitem{BCW}
Michael Batanin, Denis-Charles Cisinski, and Mark Weber.
\newblock Multitensor lifting and strictly unital higher category theory.
\newblock {\em Theory Appl. Categ.}, 28:No. 25, 804--856, 2013.

\bibitem{Bat}
Mikhail~A. Batanin.
\newblock Homotopy coherent category theory and {$A_\infty$}-structures in
  monoidal categories.
\newblock {\em J. Pure Appl. Algebra}, 123(1-3):67--103, 1998.

\bibitem{BKP}
R.~Blackwell, G.~M. Kelly, and A.~J. Power.
\newblock Two-dimensional monad theory.
\newblock {\em J. Pure Appl. Algebra}, 59(1):1--41, 1989.

\bibitem{BH}
Andrew~J. Blumberg and Michael~A. Hill.
\newblock Operadic multiplications in equivariant spectra, norms, and
  transfers.
\newblock {\em Adv. Math.}, 285:658--708, 2015.

\bibitem{BP}
P.~Bonventre and L.A. Pereira.
\newblock Genuine equivariant operads.
\newblock Available as arXiv:1707.02226.

\bibitem{CG}
A.S. Corner and N~Gurski.
\newblock Operads with general groups of equivariance and some -categorical
  aspects of operads in cat.
\newblock Available as arXiv:1312.5910v2.

\bibitem{DS}
Brian Day and Ross Street.
\newblock Lax monoids, pseudo-operads, and convolution.
\newblock In {\em Diagrammatic morphisms and applications ({S}an {F}rancisco,
  {CA}, 2000)}, volume 318 of {\em Contemp. Math.}, pages 75--96. Amer. Math.
  Soc., Providence, RI, 2003.

\bibitem{GM2}
B.~J. Guillou and J.~P. May.
\newblock Models of {$G$}-spectra as presheaves of spectra.
\newblock Algebraic \&Geometric Topology, accepted pending revision. Available
  as ArXiv: 1110.3571.

\bibitem{GM3}
B.~J. Guillou and J.~P. May.
\newblock Equivariant iterated loop space theory and permutative
  {G}--categories.
\newblock {\em Algebr. Geom. Topol.}, 17(6):3259--3339, 2017.

\bibitem{GMMO}
B.~J. Guillou, J.~P. May, M.~Merling, and A.~M. Osorno.
\newblock A symmetric monoidal and equivariant {S}egal infinite loop space
  machine.
\newblock {\em J. Pure Appl. Algebra}, 223(6):2425--2454, 2019.

\bibitem{GMM}
Bertrand~J. Guillou, J.~Peter May, and Mona Merling.
\newblock Categorical models for equivariant classifying spaces.
\newblock {\em Algebr. Geom. Topol.}, 17(5):2565--2602, 2017.

\bibitem{GW}
Javier~J. Guti\'{e}rrez and David White.
\newblock Encoding equivariant commutativity via operads.
\newblock {\em Algebr. Geom. Topol.}, 18(5):2919--2962, 2018.

\bibitem{HH}
M.A. Hill and M.J. Hopkins.
\newblock Equivariant symmetric monoidal structures.
\newblock Available as arXiv:1610,03114v1.

\bibitem{Is}
John~R. Isbell.
\newblock On coherent algebras and strict algebras.
\newblock {\em J. Algebra}, 13:299--307, 1969.

\bibitem{Kelly}
G.~M. Kelly.
\newblock Coherence theorems for lax algebras and for distributive laws.
\newblock In {\em Category {S}eminar ({P}roc. {S}em., {S}ydney, 1972/1973)},
  pages 281--375. Lecture Notes in Math., Vol. 420. Springer, Berlin, 1974.

\bibitem{Lack}
Stephen Lack.
\newblock Codescent objects and coherence.
\newblock {\em J. Pure Appl. Algebra}, 175(1-3):223--241, 2002.
\newblock Special volume celebrating the 70th birthday of Professor Max Kelly.

\bibitem{Lein}
Tom Leinster.
\newblock {\em Higher operads, higher categories}, volume 298 of {\em London
  Mathematical Society Lecture Note Series}.
\newblock Cambridge University Press, Cambridge, 2004.

\bibitem{LMS}
L.~G. Lewis, Jr., J.~P. May, and M.~Steinberger.
\newblock {\em Equivariant stable homotopy theory}, volume 1213 of {\em Lecture
  Notes in Mathematics}.
\newblock Springer-Verlag, Berlin, 1986.
\newblock With contributions by J. E. McClure.

\bibitem{CatWork}
Saunders Mac~Lane.
\newblock {\em Categories for the working mathematician}, volume~5 of {\em
  Graduate Texts in Mathematics}.
\newblock Springer-Verlag, New York, second edition, 1998.

\bibitem{MayGeo}
J.~P. May.
\newblock {\em The geometry of iterated loop spaces}.
\newblock Springer-Verlag, Berlin, 1972.
\newblock Lectures Notes in Mathematics, Vol. 271.

\bibitem{MayPerm}
J.~P. May.
\newblock {$E_{\infty }$} spaces, group completions, and permutative
  categories.
\newblock In {\em New developments in topology ({P}roc. {S}ympos. {A}lgebraic
  {T}opology, {O}xford, 1972)}, pages 61--93. London Math. Soc. Lecture Note
  Ser., No. 11. Cambridge Univ. Press, London, 1974.

\bibitem{MayPerm2}
J.~P. May.
\newblock The spectra associated to permutative categories.
\newblock {\em Topology}, 17(3):225--228, 1978.

\bibitem{Rant1}
J.~P. May.
\newblock What precisely are {$E_\infty$} ring spaces and {$E_\infty$} ring
  spectra?
\newblock In {\em New topological contexts for {G}alois theory and algebraic
  geometry ({BIRS} 2008)}, volume~16 of {\em Geom. Topol. Monogr.}, pages
  215--282. Geom. Topol. Publ., Coventry, 2009.

\bibitem{MMO}
J.~P. May, M.~Merling, and A.~M. Osorno.
\newblock Equivariant infinite loop space theory. {T}he space level story.
\newblock Memoirs AMS. Accepted pending revision. Available as
  arXiv:1704.03413.

\bibitem{Power}
A.~J. Power.
\newblock A general coherence result.
\newblock {\em J. Pure Appl. Algebra}, 57(2):165--173, 1989.

\bibitem{Emily}
Emily Riehl.
\newblock {\em Category theory in context}.
\newblock Courier Dover Publications, 2017.

\bibitem{Rubin1}
J.~Rubin.
\newblock Combinatorial ${N}_\infty$ operads.
\newblock Available as arXiv:1705.03585.

\bibitem{Rubin2}
J.~Rubin.
\newblock Normed symmetric monoidal categories.
\newblock Available as arXiv:1708.04777.

\bibitem{Seg}
Graeme Segal.
\newblock Categories and cohomology theories.
\newblock {\em Topology}, 13:293--312, 1974.

\bibitem{StMon}
Ross Street.
\newblock The formal theory of monads.
\newblock {\em J. Pure Appl. Algebra}, 2(2):149--168, 1972.

\end{thebibliography}

\end{document}